\documentclass[11pt]{article}
\usepackage[legalpaper, margin=1.2in]{geometry}
\usepackage[utf8]{inputenc}
\usepackage[english]{babel}
\usepackage{graphicx}
\usepackage{amsthm}
\usepackage{amsmath}
\usepackage{color}
\usepackage{xcolor}
\usepackage{amsfonts}
\usepackage{hyperref}
\usepackage{algorithm}
\usepackage{enumerate}
\usepackage{authblk}
\usepackage{caption}
\usepackage{subcaption}
\usepackage{mathtools}
\mathtoolsset{showonlyrefs}
\definecolor{geyellow}{RGB}{244,218,64}
\definecolor{gered}{RGB}{200,16,46}
\definecolor{geblu}{RGB}{0,38,119}
\definecolor{gecyano}{RGB}{25,155,252}
\definecolor{gedarkgray}{RGB}{37, 40, 42}
\definecolor{gelightgray}{RGB}{217, 217, 214}
\definecolor{white}{cmyk}{0, 0, 0, 0}
\definecolor{black}{cmyk}{0, 0, 0, 1}
\definecolor{myblu}{RGB}{118, 177, 237}
\definecolor{mybackgroung}{RGB}{210, 230, 249}
\definecolor{link}{HTML}{0B5A9D}
\definecolor{forestgreen}{RGB}{0, 153, 0}
\definecolor{brown}{RGB}{102, 74, 0}
\definecolor{uca_azzurro}{RGB}{0,127,163}

\newcommand{\R}{\mathbb{R}}

\newcommand{\diff}{\mathrm{d}}
\newcommand{\sign}{\operatornamewithlimits{sign}}
\newcommand{\argmin}{\operatornamewithlimits{arg\ min}}
\newcommand{\esssup}{\operatornamewithlimits{ess\ sup}}
\newcommand{\essinf}{\operatornamewithlimits{ess\ inf}}

\newcommand{\lp}{{L^p(\Omega)}} 
\newcommand{\lpvar}{{L^{p(\cdot)}(\Omega)}} 
\newcommand{\lpvarb}{{L^{p(\cdot)}}} 
\newcommand{\lqvar}{{L^{p'(\cdot)}(\Omega)}} 
\newcommand{\lqvarb}{{L^{p'(\cdot)}}} 
\newcommand{\lpvardual}{{{(L^{p(\cdot)}(\Omega))^*}}} 
\newcommand{\lpvarspass}{{{\mathcal{A}(L^{p'(\cdot)}(\Omega))}}}
\newcommand{\pnorm}{_{\lpvarb}}
\newcommand{\qnorm}{_{\lqvarb}}
\newcommand{\dualnorm}{_{(\lpvarb)^*}}
\newcommand{\spassnorm}{'_{p'(\cdot)}}
\newcommand{\rhopvar}{\rho_{p(\cdot)}}
\newcommand{\rhobarpvar}{\bar{\rho}_{p(\cdot)}}
\newcommand{\rhoqvar}{\rho_{p'(\cdot)}}

\newcommand{\jrho}{\mathbf{J}_{\rho_{p(\cdot)}}}
\newcommand{\jrhobar}{\mathbf{J}_{\bar{\rho}_{p(\cdot)}}}
\newcommand{\rhop}{\rho_{p}}

\newcommand{\dbreg}{\mathcal{D}_{\rhopvar}}

\newcommand{\xk}{x^{k}}
\newcommand{\xkk}{x^{k+1}}
\newcommand{\JpX}{\mathbf{J}_\mathcal{X}^\texttt{p}}

\newcommand{\JrX}{\mathbf{J}_\mathcal{X}^\texttt{r}}
\newcommand{\Jppiuppiu}{\mathbf{J}_{p_+}^{p_+}}

\newtheorem{theorem}{Theorem}[section]
\newtheorem{definition}{Definition}[section]

\newtheorem{lemma}{Lemma}[section]
\newtheorem{prop}{Proposition}[section]
\newtheorem{ass}{Assumption}[section]

\newtheorem{oss}{Remark}[section]

\makeatletter
\g@addto@macro\bfseries{\boldmath}
\makeatother

\title{
Modular-proximal gradient algorithms  in variable exponent Lebesgue spaces 
\thanks{LC acknowledges the support received by the H2020 RISE project NoMADS, and by the CNRS IEA project VaMOS. ML and LC thank the ANR MICROBLIND project. The work of CE is partially supported by GNCS-INDAM.}
}
\author[1,2]{Marta Lazzaretti\thanks{lazzaretti@dima.unige.it}}
\author[2]{Luca Calatroni\thanks{calatroni@i3s.unice.fr}}
\author[1]{Claudio Estatico\thanks{estatico@dima.unige.it}}
\affil[1]{Department of Mathematics, Via Dodecaneso 35, University of Genova, Italy}
\affil[2]{CNRS, UCA, Inria, Laboratoire I3S, Sophia-Antipolis, 06903 France}

\date{}

\begin{document}
\maketitle

\abstract{ 
  We consider structured optimization problems defined in terms of the sum of a smooth and convex function and a proper, lower semicontinuous (l.s.c.), convex (typically nonsmooth) function in reflexive variable exponent Lebesgue spaces $\lpvar$.Due to their intrinsic space-variant properties, such spaces can be naturally used as solution spaces and  combined with space-variant functionals for the solution of ill-posed inverse problems. For this purpose, we propose and analyze two instances (primal and dual) of proximal gradient algorithms in $\lpvar$, where the proximal step, rather than depending on the natural (non-separable) $\lpvar$ norm, is defined in terms of its modular function, which, thanks to its separability, allows for the efficient computation of algorithmic iterates. Convergence in function values is proved for both algorithms, with convergence rates depending on problem/space smoothness. To show the effectiveness of the proposed modeling,  some numerical tests highlighting the flexibility of the space $\lpvar$ are shown for exemplar deconvolution and mixed noise removal problems. Finally, a numerical verification of the convergence speed and computational costs of both algorithms in comparison with analogous ones defined in standard $\lp$ spaces is  presented.
}

\section{Introduction}
\label{sec:intro}
Let $\Omega \subseteq \mathbb R^d$, with $d\in\mathbb N,\ d\ge1$ be a Lebesgue measurable subset with positive measure and let $p(\cdot):\Omega \longrightarrow [1,+\infty]$ be a Lebesgue measurable function. The  variable exponent  Lebesgue space $\lpvar$ is defined in terms of the function $p(\cdot)$ which, differently from conventional $\lp$ spaces, acts as pointwise variable exponent, inducing a specific shift-variant norm. In this work, we will study a proximal-gradient (a.k.a. forward-backward) splitting algorithm to solve composite optimization problems in the form 
\begin{equation}\tag{P} \argmin_{x\in\lpvar} \phi(x):= f(x)+g(x) 
\label{eq:P}
\end{equation}
where $f:\lpvar\longrightarrow \mathbb R \cup \{+\infty\}$ is a proper, convex, and Gateaux differentiable function while $g:\lpvar\longrightarrow \mathbb R \cup \{+\infty\}$ is lower semicontinuous (l.s.c.), proper, convex, and possibly nonsmooth. Optimization problems of this kind often arise in the context of variational regularization of inverse problems such as, e.g., the minimization of Tikhonov-like functionals, where $f$ denotes a data-fidelity term while $g$ stands for the regularization or penalty term, see, e.g. \cite{chambolle_pock_2016} for a review. 

Variable exponent Lebesgue spaces $\lpvarb$ are (non-Hilbertian) Banach spaces, endowed with space-variant geometrical properties which are useful in adaptive regularization techniques, see, e.g., \cite{Brander2019, Briceno,GhilliLorenz2021,Lorenz2016}. Under suitable choices of the exponent function $p(\cdot)$, having $\lpvarb$ as solution space allows for a better recovery of sparse solutions, avoiding oversmoothing reconstruction artifacts of usual regularization models formulated in $L^2(\Omega)$, a property that was already observed in previous works (see, e.g., \cite{Estatico2018}) in the context of iterative regularization methods (with no penalty).

Forward-backward splitting algorithms are a common strategy used to solve problems in the form \eqref{eq:P}. Initially proposed in \cite{ChenRockafellar1997,CombettesWajs2005,LionsMercier1979} in Hilbert spaces $\mathcal{H}$ for functions $f$  with $\mathcal{L}$-Lipschitz gradient, they are defined by the following updating rule for $k\geq 0$
\[
\xkk = \argmin_{x\in\mathcal{H}}\frac{1}{2}\|x-(\xk-\tau_k\nabla f(\xk))\|_{\mathcal{H}}^2+\tau_k g(x),
\]
where $\tau_k>0$ is a suitably chosen sequence of step sizes. Such schemes enable one to exploit the differentiability of the smooth function $f$ in a forward step defined in terms of the gradient of $f$ $\xk\mapsto \xk-\tau_k\nabla f(\xk)$ which is decoupled from the backward step defined in terms of the proximal operator associated to $g$.

The generalization of this algorithm to a Banach space $\mathcal{X}$ is not straightforward since the element $\nabla f(\xk)\in \mathcal{X}^*$ cannot be identified as an element of $\mathcal{X}$ anymore due to the lack of Riesz isomorphism.  As a consequence, the forward step cannot be performed directly on $\xk\in \mathcal{X}$. As a remedy,  we need to introduce the so-called duality maps, which link primal and dual spaces and allow us to perform the forward gradient step either in the dual or in the primal space. The intrinsic nonlinearity of duality mappings, however, introduces new challenges in the definition of a backward step in terms of the proximal operator of $g$. In \cite{Bredies2008,GuanSong2015,GuanSong2021}, forward-backward algorithms have been proposed to solve minimization problems in the form \eqref{eq:P} and defined in a reflexive, strictly convex and smooth Banach space $\mathcal{X}$, by means of suitably defined notions of duality mappings and proximal operators. 

In particular, in \cite{Bredies2008} Bredies introduces the following iterative procedure for smooth functions $f$ with $(\texttt{p}-1)$-H{\"o}lder continuous gradient $\nabla f$  on bounded sets with $1<\texttt{p}\le2$ with constant $K>0$ and step-sizes chosen as $0<\bar{\tau}\le \tau_k \le \frac{\texttt{p}(1-\delta)}{K}$, with $0<\delta<1$:
 \begin{equation}
\xkk \in\argmin _{x \in \mathcal{X}} \frac{1}{\texttt{p}}\|x-\xk\|_\mathcal{X}^{\texttt{p}}+\tau_k\left\langle \nabla f(\xk), x\right\rangle+\tau_k g(x)\,,
\label{brediesiter}
\end{equation}
where $\left\langle u^*,x\right\rangle=\left\langle x,u^*\right\rangle=u^*(x) \in \mathbb R$ denotes the duality pairing, for $u^*\in\mathcal{X}^*$ and $x\in\mathcal{X}$. Notice that, for $g\equiv 0$, the  updating rule \eqref{brediesiter} can be also written as
$$  0 \in \partial\Big(\frac{1}{\texttt{p}}\| \cdot -\xk\|_\mathcal{X}^{\texttt{p}}+\tau_k\left\langle \nabla f(\xk), \cdot \right\rangle \Big)(\xkk),$$
where, for a convex function $h:\mathcal{X}\longrightarrow \mathbb{R}$, $\partial h(w)$ denotes the subdifferential of $h$ evaluated at $z$. By linearity, this leads to the following equivalent inclusions:
 \begin{align}
    0 \in\JpX(\xkk-\xk)+\tau_k \nabla f(\xk)  
    \quad  \Leftrightarrow \quad
    \xkk \in\xk-\tau_k \mathbf{J}_{\mathcal{X}^*}^{\texttt{p}'}(\nabla f(\xk)),   \label{eq:primal_algo}
    \end{align}
where $\JpX: \mathcal{X} \longrightarrow {\mathcal{X}^*}$, $\JpX
= \partial\Big(\frac{1}{\texttt{p}}\|\cdot\|_\mathcal{X}^{\texttt{p}}\Big)$, is the $\texttt{p}$-duality map of the (primal) space $\mathcal{X}$ and $\mathbf{J}_\mathcal{X^*}^{\texttt{p}'}: {\mathcal{X}^*} \longrightarrow \mathcal{X}\,$ is its inverse, that is, the $\texttt{p}'$-duality map of the dual space  $\mathcal{X^*}$, with $\texttt{p}'$ the H{\"o}lder conjugate of $\texttt{p}$ \cite{Cioranescu1990}. 
Since $ \mathbf{J}_\mathcal{X^*}^{\texttt{p}'}(\nabla f(\xk)) \in \mathcal{X}$, the latter inclusion shows that \eqref{brediesiter} reduces to a gradient descent step performed in the primal space $\mathcal{X}$. It can thus be seen a generalization of the  \textit{primal method} described in \cite{Schuster2012} to solve inverse problems $Ax=y \in\mathcal{Y}$, with $A:\mathcal{X}\longrightarrow \mathcal{Y}$ linear operator between two Banach spaces $\mathcal{X}\ni x$ and $\mathcal{Y}$, via the minimization of the residual functional $f: \mathcal{X} \longrightarrow \mathbb R$ defined as $f(x)=\frac{1}{\texttt{p}}\|Ax-y\|_\mathcal{Y}^\texttt{p}$. 
Such interpretation is not obvious from \eqref{brediesiter}, where forward and backward steps are defined as one single minimization problem, so that they cannot be distinguished. Moreover, it shows the tight link between proximal-gradient schemes and regularization theory in Banach spaces  \cite{Schopfer2006,Schuster2012}, notions that relate to the fields of optimization/convex analysis and functional analysis, respectively.

In \cite{GuanSong2015}, the authors considered a different forward-backward splitting algorithm in Banach spaces, where the proximal step depends on Bregman distance as follows
\begin{equation}
\xkk = \argmin_{x \in \mathcal{X}}\frac{1}{\texttt{p}'}\|\xk\|_\mathcal{X}^{\texttt{p}}+\frac{1}{\texttt{p}}\|x\|_\mathcal{X}^{\texttt{p}}-\left\langle \JpX(\xk), x\right\rangle+\tau_k\left\langle\nabla f(\xk), x\right\rangle+\tau_kg(x),
\label{guansongiter}
\end{equation}
with $f$ as above and suitably chosen step-sizes (see \cite{GuanSong2015}). In this second case, the algorithm requires the computation of the forward step in the dual space. Similarly as before, for $g\equiv 0$ we notice that the updating rule \eqref{guansongiter} can now be written as
$$  0 \in \partial\Big(\frac{1}{\texttt{p}'}\|\xk\|_\mathcal{X}^{\texttt{p}}+\frac{1}{\texttt{p}}\|\cdot\|_\mathcal{X}^{\texttt{p}}-\left\langle \JpX(\xk), \cdot\right\rangle+\tau_k\left\langle\nabla f(\xk), \cdot\right\rangle \Big)(\xkk),$$
so that we recover a generalization of the so-called dual method \cite{Schuster2012}
\begin{equation}  \label{eq:dual_algo}
    \JpX(\xkk) \in \JpX(\xk)-\tau_k \nabla f(\xk) \qquad
    \xkk=\mathbf{J}_\mathcal{X^*}^{\texttt{p}'}\Big(\JpX(\xkk)\Big),
\end{equation}
introduced in \cite{Schopfer2006} for $f(x)=\frac{1}{\texttt{p}}\|Ax-y\|_\mathcal{Y}^\texttt{p}$, where, differently from \eqref{brediesiter}, the forward step is now computed in $\mathcal{X}^*$, since we have $\JpX(\xk), \nabla f(\xk) \in \mathcal{X^*}$.

Algorithms \eqref{brediesiter} and \eqref{guansongiter} have been proposed for reflexive, strictly convex and smooth Banach spaces, hence they can a priori be applied to solve \eqref{eq:P} in $\lpvar$ whenever $1<p(t)<+\infty \text{ a.e. in } \Omega$. However, the definitions of  $\|\cdot\|\pnorm$ and of duality mappings make their use impracticable in real applications, since both are not separable, differently from what happens, e.g., in $\lp$ spaces. 
To overcome this issue, we propose here to replace the role of the norm naturally appearing in the definition of the proximal operator by the modular functions defined by
\begin{equation*}
	\rhobarpvar(x):=\int_\Omega \frac{1}{p(t)}|x(t)|^{p(t)} \diff t,   \qquad \rhopvar(x):=\int_\Omega |x(t)|^{p(t)} \diff t \quad \forall \ x\in\lpvar.
\end{equation*}
Due to their additive separability, they allow efficient computations of proximal points.

On these grounds, inspired by \cite{Bredies2008} we study an iterative proximal-gradient algorithm whose iteration for $k\geq 0$ reads as
\[
	\xkk=\argmin_{x\in\lpvar} \rhobarpvar(x-\xk) + \tau_k \langle \nabla f(\xk),x\rangle+\tau_k g(x)
	\]
for a proper choice of the sequence $\tau_k>0$. 
Following \cite{GuanSong2015}, we then propose another proximal-gradient algorithm whose update reads for $k\geq 0$ as
\[
\xkk = \argmin_{x\in\lpvar} \rhobarpvar(x)-\langle \jrhobar(\xk), x\rangle  + \tau_k \langle \nabla f(\xk),x\rangle+\tau_k g(x),
\]
where $\jrhobar$ stands for the gradient of the modular $\rhobarpvar$. 
For both algorithms, we prove convergence in function values and provide an interpretation as primal and dual algorithms, respectively. We next show how the intrinsic space-variant properties of the underlying space $\lpvar$ correspond to the use of adaptive nonstandard thresholding functions which turn out to be a flexible tool in the solution of exemplar deconvolution and denoising problems, as we show in the numerical section. Finally, a numerical verification of the algorithmic convergence rates, in comparison with standard Hilbert and $\lp$-type algorithms is performed.


\section{Preliminaries on variable exponent Lebesgue spaces}
\label{sec:Preliminaries}

We introduce in this section some definitions, basic rationales and key concepts of $\lpvar$ spaces. For a thorough survey of these spaces, see \cite{LpvarBOOK,CruzUribeFiorenzaBOOK}.

\subsection{Modular and Luxemburg norm}

Let the set of all possible exponents be $\mathcal{P}(\Omega):=\{p(\cdot):\Omega \longrightarrow [1,+\infty]\ | \ p(\cdot) \text{ is Lebesgue measurable}\}$. Given $p(\cdot)\in\mathcal{P}(\Omega)$, we denote the essential infimum and essential supremum of $p(\cdot)$ by
\[
p_-:=\essinf_{u\in\Omega} p(u) \qquad \text{and} \qquad p_+:=\esssup_{u\in\Omega} p(u).
\]
For the sake of simplicity, we assume in the following $ 1<p_-\le p_+<+\infty$ (for the general case, see \cite{LpvarBOOK}). We will denote by $\mathcal{F}(\Omega)$ the set of all Lebesgue measurable functions $x:\Omega\longrightarrow\R\cup\{+\infty\}$. In order to characterize spaces $\lpvar$, it is first necessary to introduce the key concept of modular function.
\begin{definition} 
Given an exponent $p(\cdot)\in\mathcal{P}(\Omega)$ with $p_+<+\infty$, the function $\rhopvar:\mathcal{F}(\Omega)\longrightarrow[0,+\infty]$ defined by
\begin{equation}
    \rhopvar(x)=\int_\Omega |x(t)|^{p(t)} \diff t
\label{eq:mod_rho}
\end{equation}
is called modular associated to the exponent function $p(\cdot)$. An alternative definition of modular function consists in considering $\rhobarpvar:\mathcal{F}(\Omega)\longrightarrow[0,+\infty]$ defined by
\begin{equation}
    \rhobarpvar(x)=\int_\Omega \frac{1}{p(t)}|x(t)|^{p(t)} \diff t.
\label{eq:mod_rhobar}
\end{equation}
\end{definition}
We will refer to \eqref{eq:mod_rho} and \eqref{eq:mod_rhobar} as modular functions, as needed. Notice that $\rhopvar(x)$ is the generalization of the $p$-power $\|x\|_{p}^p=\int_\Omega |x(t)|^p\diff t$ in $\lp$ with constant exponent $ p\in (1,+\infty)$. Similarly, $\rhobarpvar(x)$ generalizes the quantity $\frac{1}{p}\|x\|_{p}^p$.  The modular allows us to characterize the variable exponent space $\lpvarb$ and to define its norm, in the general framework of the Luxemburg norms of Orlicz spaces \cite{LpvarBOOK}.

\begin{definition}
The space $\lpvar$ is the set of functions $x\in\mathcal{F}(\Omega)$ such that
$
\rhopvar\Big(\frac{x}{\lambda}\Big)\le 1,
$ 
for some $\lambda>0$. For any $x\in\lpvar$, we define $\|\cdot\|_{\lpvar}:\lpvar\longrightarrow\R$ as
\begin{equation}
    \|x\|\pnorm:=\inf\left\{\lambda>0:\rhopvar\Big(\frac{x}{\lambda}\Big)\le1\right\}.
\label{eq:Lux_norm}
\end{equation}
\end{definition}
\begin{theorem} \cite{LpvarBOOK}
The function $\|\cdot\|_{\lpvar}$ defined in \eqref{eq:Lux_norm} is a norm on $\lpvar$ and the space $\Big (\lpvar, \|\cdot\|_{\lpvar} \Big )$ is a Banach space.
\end{theorem}

It is interesting to point out that if $p(\cdot)\equiv p\in(1,+\infty)$,  the classical notion of norm $\|x\|_p$ in $\lp$ can be retrieved. Indeed, $\rhopvar\Big(\frac{x}{\lambda}\Big)=\rhop\Big(\frac{x}{\lambda}\Big)=\frac{1}{\lambda^p}\rhop(x)= \frac{1}{\lambda^p}\|x\|_p^p$
so that the infimum in \eqref{eq:Lux_norm} is equal to $\|x\|_p$. We remark that the computation of the $p$-radical of the integral is necessary to ensure the homogeneity property $\|\alpha x\|_p=|\alpha| \|x\|_p$, for any $\alpha \in \mathbb C$. With a variable exponent, such computation is obviously not possible and in turn the one-dimensional (1D) minimization problem \eqref{eq:Lux_norm} has to be solved. However, the quantity \eqref{eq:Lux_norm} can be bounded by the $p_-$ and $p_+$ radicals  of the modular. At a certain extent, one can thus think of the norm as $\Tilde{p}$-radical (which depends on $x$) of the modular with $p_-\le\Tilde{p}\le p_+$, as stated by the following lemma.
\begin{lemma}\cite[Lemma 3.2.5, Lemma 3.4.2]{LpvarBOOK}
Let $p(\cdot)\in\mathcal{P}(\Omega)$ with $p_+<+\infty$. 
\begin{enumerate}[(i)]
    \item If $\|x\|\pnorm>1$, then $\rhopvar(x)^{1/p_+}\le\|x\|\pnorm\le\rhopvar(x)^{1/p_-}.$ \label{eq:normamaggioredi1}
    \item If $0<\|x\|\pnorm\le1$, then $\rhopvar(x)^{1/p_-}\le\|x\|\pnorm\le\rhopvar(x)^{1/p_+}.$ \label{eq:normaminoredi1}
    \item $\|x\|\pnorm<1$ and $\rhopvar(x)< 1$ are equivalent.
    \label{normamodular2_i}
    \item $\|x\|\pnorm=1$ and $\rhopvar(x) = 1$ are equivalent.
    \label{normamodular2_ii}
\end{enumerate}
\label{normamodular2}
\end{lemma}
If $1<p_-\le p_+<+\infty$, then the following useful properties hold true.
\begin{theorem}\cite[Theorem 3.4.7, Theorem 3.4.9]{LpvarBOOK} \cite[Lemma 1]{DincaMatei2009lungo} \label{lpvar:reflexivity}
Given $p(\cdot)$ such that $1<p_-\le p_+<+\infty$, then $\lpvar$ is reflexive, uniformly convex and smooth.
\end{theorem}

\subsection{Dual space, duality mappings, separability}

In $\lp$ spaces, the isometric isomorphism between the space $\Big ( \lp \Big )^*$ and $L^{p'}(\Omega)$, with $p'$ being the H\"{o}lder conjugate of $p$, has a key role in devising regularization algorithms \cite{Schuster2012}. In $\lpvar$ spaces, such isomorphism does not hold true, as we briefly summarize in this subsection. For a comprehensive description of these arguments, see \cite{LpvarBOOK}.

\begin{definition}
Let $G:\lpvar\longrightarrow\R$ be a linear functional. $G$ is bounded if 
$
\sup_{u\in\lpvar,\|u\|\pnorm \le 1} |G(u)|<+\infty.
$
The dual space of $\lpvar$ can thus be defined as the set 
\[
    \lpvardual=\{G:\lpvar\longrightarrow\R \ : \ G \ \text{is linear and bounded}\} \,,
\]
which is a Banach space with the norm
\[
    \|G\|_\lpvardual:=\sup_{u\in\lpvar,\|u\|\pnorm \le 1} |G(u)|.
\]
\end{definition}
For $1<p_-\le p_+<+\infty$, the H{\"o}lder conjugate of $p(\cdot)$ is 
a Lebesgue measurable function $p'(\cdot)\in\mathcal{P}(\Omega)$ such that
$
\frac{1}{p(t)}+\frac{1}{p'(t)}=1 
$ a.e. in $\Omega$.
For any $z\in\lqvar$, it can be shown that there exists a unique $G\in\lpvardual$ such that 
\[
    G(u)=\int_\Omega z(t)u(t)\diff t \qquad \forall\ u\in\lpvar.
\]
Thus, we can denote unambiguously $G$ as $G_z$.
\begin{definition}\cite[Definition 2.7.1]{LpvarBOOK}
The associate space of $\lpvar$, denoted by $\lpvarspass$, is the space of functions $z\in\lqvar$ such that 
\begin{equation}
    \|z\|\spassnorm:=\sup_{u\in\lpvar,\|u\|\pnorm \le 1} ~\int_\Omega |z(t)||u(t)|\diff t<+\infty.
    \label{normaspazioass}
\end{equation}
\end{definition}
It can be proved that $\|\cdot\|\spassnorm$ is a norm on $\lpvarspass$. Note that since
\begin{equation}
\begin{aligned}
\|G_z\|_\lpvardual&=\sup_{\|u\|\pnorm \le 1} |G_z(u)|=\sup_{\|u\|\pnorm \le 1}\int_\Omega |z(t)||u(t)|\diff t=\|z\|\spassnorm \,,
\end{aligned}
\label{eq:isom_ass}
\end{equation}
\eqref{normaspazioass} is finite if and only if the linear operator $G_z$ is bounded. Thus, there exists an isometric embedding between $\lpvarspass$ and $\lpvardual$.
\begin{prop}\cite[Corollary 3.2.14]{LpvarBOOK}
\label{prop:normaspassduale}
For all $z\in\lpvarspass$, there holds
\[
    \frac{1}{2}\|z\|\qnorm\le\|z\|\spassnorm\le 2\|z\|\qnorm\,,
\]
and the bounds are optimal.
\end{prop}
\begin{oss}
From the proposition above combined with \eqref{eq:isom_ass}, it follows that $\lpvardual$ and $\lqvar$ are not isometrically isomorphic. 
\end{oss}
A key role in the definition of iterative schemes in Banach space $\mathcal{X}$ is played by duality maps, which associate an element of $\mathcal{X}$ with a specific element of its dual $\mathcal{X}^*$. 

\begin{definition}\cite{Cioranescu1990}
\label{def:map_dual}
Let $\mathcal{X}$ be a Banach space and $r>1$. Then the duality map $\JrX$ with gauge function $t\mapsto t^{r-1}$ is the operator $\JrX:\mathcal{X}\to 2^{\mathcal{X}^*}$ such that
$$
\JrX(x)=\big\{x^{*}\in \mathcal{X}^{*}\mid x^{*}(x)=\left<x^*,x \right>=\|x\|_\mathcal{X} \|x^{*}\|_{\mathcal{X}^{*}}, \, \|x^{*}\|_{\mathcal{X}^{*}}=\|x\|_\mathcal{X}^{r-1}\big\}\quad \forall x\in X.
$$
\end{definition}

If $\mathcal{X}$ is smooth, then the duality map is single valued, that is $\JrX:\mathcal{X}\to \mathcal{X}^*$. In addition, if $\mathcal{X}$ is a Hilbert space $\mathcal{H}$, by virtue of the Riesz theorem, the duality map reduces to the identity operator for $r=2$, that is  $\mathbf{J}_\mathcal{H}^\texttt{2}(x)=x$, where the isometric isomorphism between $\mathcal{H}$ and $\mathcal{H}^*$ has been implicitly considered.

In general, the following result gives a more intuitive interpretation of the important role played by duality mappings in the definition of minimization schemes.
\begin{theorem} [(Asplund) \cite{Cioranescu1990}]
\label{Theo:Asplund}
Let $\mathcal{X}$ be a Banach space, and let $r>1$. The $r$-duality map $\JrX$ is the subdifferential of the convex functional $h:\mathcal{X} \longrightarrow \mathbb R$, $h(x)=\frac{1}{r}\|x\|_\mathcal{X}^r$:
$$ \JrX = \partial h = \partial \left ( \frac{1}{r}\| \cdot \|_\mathcal{X}^r \right )\,.$$
\end{theorem}
Thanks to this Theorem, the duality maps are thus monotone operators.

In \cite{DincaMatei2009corto,DincaMatei2009lungo} the authors proved that  $\|\cdot\|\pnorm$ is Gateaux-differentiable, providing an analytical expression for its Gateau derivative, and in \cite{Matei2012,Matei2014} that it is Fréchet differentiable too, for any $x\not =0$. From these results, it follows that $\frac{1}{r}\|\cdot\|\pnorm^r$ for $r>1$ is Fréchet differentiable for any $x \in \mathcal{X}$, hence its Gateaux derivative, the $r$-duality map $\mathbf{J}_{\lpvarb}^r$, can be obtained following arguments similar to those of \cite{DincaMatei2009corto,DincaMatei2009lungo}. 
\begin{prop} 
For each $x,u\in\lpvar$ and for any $r\in(1,+\infty)$, the duality mapping $\mathbf{J}_{\lpvarb}^r:\lpvar\longrightarrow\lpvardual$ is a linear operator with expression:
\begin{equation}
    \langle\mathbf{J}_{\lpvarb}^r(x),u\rangle=\frac{1}{\int_\Omega \frac{p(t)|x(t)|^{p(t)}}{\|x\|\pnorm^{p(t)}}\diff t } \int_\Omega \frac{p(t)\sign\big(x(t)\big) |x(t)|^{p(t)-1}}{\|x\|\pnorm^{p(t)-r}}u(t)\diff t.
\label{def:dual_map_lpvar}
\end{equation}
\end{prop}

If $p(\cdot)\equiv p$ is constant, then $\mathbf{J}_{\lpvarb}^r$ coincides with $\mathbf{J}_p^r$, duality map of $L^p(\Omega)$:
\[
    \langle\mathbf{J}_p^r(x),u\rangle=\|x\|_p^{r-p} \int_\Omega \sign\big(x(t)\big)|x(t)|^{p-1}u(t)\diff t.
\]
\begin{prop}
\label{prop_jrho_jrhobar}
For any $x, u\in\lpvar$, with $p_+<+\infty$, the functions $\rhopvar(\cdot)$ and $\rhobarpvar(\cdot)$ are Gateaux differentiable, with derivatives 
\begin{equation}
\langle\jrho(x),u\rangle=\int_\Omega p(t)\sign(x(t))|x(t)|^{p(t)-1}u(t)\diff t
\end{equation}
\begin{equation}
\langle\jrhobar(x),u\rangle=\int_\Omega \sign(x(t))|x(t)|^{p(t)-1}u(t)\diff t.
\end{equation}
\end{prop}
\begin{proof}
Let $x,u\in\lpvar$. The Gateaux derivative of $\rhopvar$ at $x$ along direction $u$ is given by
\begin{equation*}
    \lim_{t\to 0}~ \frac{\rhopvar(x+tu)-\rhopvar(x)}{t}=\Big[\frac{\partial}{\partial t}\rhopvar(x+tu)\Big]_{t=0}.
\end{equation*}
First, note that
$\rhopvar(x+tu)=\int_\Omega |x(s)+tu(s)|^{p(s)}\diff s<+\infty.
$
Indeed $x,u\in\lpvar$, then $x+tu\in\lpvar$ and consequently $\|x+tu\|\pnorm<+\infty$ by definition of $\lpvar$ space. Since $p_+<+\infty$ and by Lemma \ref{normamodular2}, $\|x+tu\|\pnorm<+\infty$ implies $\rhopvar(x+tu)<+\infty$. We can thus compute the partial derivative of $\rhopvar(x+tu)$ with respect to $t$ by ``differentiating under the integral sign.'' To do so, we first need to verify the regularity of the integrand function. Let $f:\Omega\times(-1,1)\longrightarrow\R$ be defined by
\begin{equation*}
    f(s,t):=|x(s)+tu(s)|^{p(s)},\quad s\in\Omega,\quad t\in(-1,1).
\end{equation*}
By direct computations, we obtain 
\begin{equation*}
    \frac{\partial f}{\partial t}(s,t)=p(s)|x(s)+tu(s)|^{p(s)-1}\sign\Big(x(s)+tu(s)\Big)u(s)
\end{equation*}
and, since $|t|<1$,
\begin{equation*}
\begin{aligned}
\Big|\frac{\partial f}{\partial t}(s,t)\Big|& \le p_+ |x(s)+tu(s)|^{p(s)-1} |u(s)|  \le p_+ |x(s)+tu(s)|^{p(s)-1} (|u(s)|+|x(s)|)\\
& \le p_+(|u(s)|+|x(s)|)^{p(s)}=:g(s),
\end{aligned}
\end{equation*}
with $g(s)$ integrable. We conclude the proof thanks to the dominated convergence theorem.
\end{proof}
\begin{oss}
We stress that although $\jrho$ and $\jrhobar$ are not duality mappings, we adopt nonetheless the same notation for consistency.
\end{oss}
It is interesting to observe the following property of the modular function and its gradient. By direct computations, it follows that for any $x\in\lpvar$ there holds
\begin{equation}
    \langle\jrhobar(x),x\rangle=\int_\Omega \sign(x(t)) \, |x(t)|^{p(t)-1}x(t) \, \diff t=\rhopvar(x).
    \label{eq:Jandmodular}
\end{equation}
This is the analogue of a general property of  duality mappings in Banach spaces. Indeed, by Definition \ref{def:map_dual}, if $\mathcal{X}$ is a smooth Banach space, then for any $r>1$ and $x\in\mathcal{X}$ there holds $\langle\mathbf{J}_\mathcal{X}^r(x),x\rangle = \|x\|_\mathcal{X} \|x^*\|_{\mathcal{X}^*} = \|x\|_\mathcal{X} \|x\|_\mathcal{X}^{r-1} = \|x\|_\mathcal{X}^r.$

As better explained in Section \ref{sec:Applications}, it is handy having functionals and operators defined in Banach spaces that are \emph{separable}, that is, their global computation can be decomposed into the sum of low-dimensional functionals. To make this property more precise, we consider the following definition of domain additive separability.

\begin{definition} 
\label{def:add_sep}
Let $\mathcal{X}$ be a functional Banach space on $\Omega$. An operator $S:\mathcal{X} \rightarrow \mathcal{X}^*$ or a functional $S:\mathcal{X} \rightarrow \mathbb R$ is domain additively separable, if, for any finite family of Lebesgue measurable subsets $\Big (\Omega_i \Big)_{i=1}^{n}$ of $\Omega$ such that $\mathring{\Omega}_i\cap\mathring{\Omega}_j=\emptyset$ for $i\not=j$, and $\Omega=\bigcup_{i=1}^n{\Omega_i}$, there holds 
$
S(x)=\sum_{i=1}^{n} S\left(\chi_i \, x\right)
$
for any $x \in \mathcal{X}$, where $\chi_i\in \mathcal{X}$ is the characteristic function of $\Omega_i$, that is $\chi_i(t)=1$ for  $t\in\Omega_i$, and 
$\chi_i(t)=0$ for $t\not\in\Omega_i$.
\end{definition}

In the following, domain additive separability will often be referred to simply as separability. It is quite evident that in conventional $\lp$ spaces, the norm functional $\|\cdot\|_p^p$, as well as the operator $\mathbf{J}_p^p(\cdot)$, are domain additively separable, since, for any suitable family of subsets $\Big (\Omega_i \Big)_{i=1}^{n}$, there holds $\|x\|_p^p=\sum_{i=1}^{n} \| \chi_i x\|_p^p$ and $\langle\mathbf{J}_p^p(x),u\rangle=\sum_{i=1}^{n} \langle\mathbf{J}_p^p(\chi_i x),u\rangle=\langle \sum_{i=1}^{n} \mathbf{J}_p^p(\chi_i \, x),u\rangle$.
On the contrary, norms and duality maps in variable exponent spaces are not separable.

\begin{lemma}
The norm and the duality mapping in $\lpvar$ are not domain additively separable in the sense of Definition \ref{def:add_sep}.
\label{lemma:non_sep}
\end{lemma}
\begin{proof}
It is quite evident that the Luxemburg norm \eqref{eq:Lux_norm} requires the solution of a 1D minimization problem on the entire domain $\Omega$, which, in general, cannot be divided into the solutions on single sets of the partition, that is,
$\|x\|\pnorm\not = \sum_{i=1}^{n} \| \chi_i  \, x\|_{\lpvar}\,$. 
As the duality mapping is concerned, the two norms in the denominators of $\mathbf{J}_{L^{p(\cdot)}}$ \eqref{def:dual_map_lpvar} show that its computation cannot be decomposed into the computation of $n$ integrals involving only the restriction of the function $x$ onto single sets of the partition, or, in other words, $\mathbf{J}_{\lpvar}^r(x)\not = \sum_{i=1}^{n} \mathbf{J}_{\lpvar}^r( \chi_i  \, x)\,$.
\end{proof}

The modular functions introduced in Definition \ref{eq:mod_rho} as well as their gradients turn out instead to satisfy the separability property. 

\begin{lemma}
The modular functions in $\lpvar$  and their gradients are domain additively separable, in the sense of Definition \ref{def:add_sep}.
\end{lemma}
\begin{proof}
We consider the modular function $\rhobarpvar(x)=\int_\Omega \frac{1}{p(t)}|x(t)|^{p(t)} \diff t$ defined in \eqref{eq:mod_rhobar} (for \eqref{eq:mod_rho}, the proof is similar). By direct computation, by the linearity property of the integral w.r.t. the integration domain, we have
\begin{equation*}
\begin{aligned}
\rhobarpvar(x) &=
\sum_{i=1}^{n} \int_{\Omega_i} \frac{1}{p(t)}|x(t)|^{p(t)} \diff t = \sum_{i=1}^{n} \int_\Omega \frac{1}{p(t)}|\chi_i(t)x(t)|^{p(t)} \diff t = \sum_{i=1}^{n} \rhobarpvar(\chi_i x).
\end{aligned}
\end{equation*}
Similarly, for  $\jrhobar$, we can write
\begin{equation*}
\begin{aligned}
& \langle\jrhobar(x),u\rangle  =\int_\Omega \sign(x(t))|x(t)|^{p(t)-1}u(t)\diff t = 
\sum_{i=1}^{n} \int_{\Omega_i} \sign(x(t))|x(t)|^{p(t)-1}u(t)\diff t\\
 & = \sum_{i=1}^{n} \int_{\Omega} \sign(x(t))|\chi_i(t)  \, x(t)|^{p(t)-1}u(t)\diff t = \sum_{i=1}^{n} \langle\jrhobar(\chi_i  \, x),u\rangle = \langle\sum_{i=1}^{n}\jrhobar(\chi_i  \, x),u\rangle
\end{aligned}
\end{equation*}
which concludes the proof.
\end{proof}


\section{A modular-proximal gradient algorithm}
\label{sec:Bredies_lpvar}
In this section, we propose and analyze an iterative procedure to solve the minimization problem \eqref{eq:P}.
We set $\Bar{\phi}:=\inf_{x\in\lpvar} \phi(x)$, and define Sol(P) as $\text{Sol(P)}:=\{x\in\lpvar:\phi(x)=\Bar{\phi}\}\neq \emptyset$. We consider the following two assumptions.
\begin{ass}
The exponent function $p(\cdot)$ is such that $1<p_-\le p_+\le2$.
\label{ass1_bredies}
\end{ass}
\begin{ass}
\label{holder}
$\nabla f : \lpvar \longrightarrow \lpvar^*$ is $(\textit{\texttt{p}} -1)$H{\"o}lder-continuous with exponent $p_+\le\textit{\texttt{p}}\le2$ and constant $K>0$, i.e.: 
\[
\|\nabla f(u)-\nabla f(v)\|\dualnorm \le K\|u-v\|\pnorm^{\textit{\texttt{p}} -1} \qquad \forall \ u,v\in\lpvar.
\]
\end{ass}
The first modular-proximal gradient algorithm we propose is reported in Algorithm \ref{alg_bredies}, as pseudocode. 
\begin{algorithm}[h!]
\caption{Modular-proximal gradient algorithm in $\lpvar$ spaces}
\label{alg_bredies}
\small{
\textbf{Parameters:} $\rho\in(0,1)$, $\{\tau_k\}_k$ s.t. 
\begin{equation}
0<\Bar{\tau}\le \tau_k \le \frac{\textit{\texttt{p}}(1-\delta)}{K} \qquad \text{with} \ 0<\delta<1.
\label{eq:step_bunds_algbredies}
\end{equation}
\textbf{Initialization:} Start with $x^0\in\lpvar$.\\
{
\textsc{FOR $k=0,1,\ldots$ REPEAT}}
\begin{itemize}
\item[]{\textsc{FOR $i=0,1,\ldots $ REPEAT}}
    \begin{itemize}
        \item[\textsc{1.}] Set $\tau_{k}=\rho^i \tau_{k}$.
        \item[\textsc{2.}] Compute the next iterate as: 
        \begin{equation}
        \xkk=\argmin_{x\in\lpvar} \rhobarpvar(x-\xk) + \tau_k \langle \nabla f(\xk),x\rangle+\tau_k g(x).
        \label{bredies} 
        \end{equation}
    \end{itemize}
\item[]{\textsc{UNTIL}} $\rhopvar(\xk-\xkk)<1$
\end{itemize}
{\textsc{UNTIL}} convergence
}
\end{algorithm}
The inner loop is needed to select at each $k$-th iteration a sufficiently small step-size $\tau_k$ such that $\rhopvar(\xk-\xkk)<1$, which is required in the following convergence analysis as we will see in the proofs of Proposition \ref{prop:bredies_descent} and Lemma \ref{lemma:bredies_rate}. It should be thought as a backtracking-like procedure affecting more the first algorithmic iterations where the quantity $\rhopvar(\xk-\xkk)$ is likely to be large. 

We start our analysis discussing the well-definition of the step \eqref{bredies}, e.g. the existence and uniqueness of the minimizer of the functional defined by \eqref{bredies}.
\begin{prop}  \label{prop:bredies1}
For each $x\in\lpvar$, $v^*\in\lpvar^*$ and $\tau>0$, the problem
\begin{equation}
\argmin_{u\in\lpvar} \rhobarpvar(u-x) + \tau \langle v^*,u\rangle+\tau g(u)
\label{funct_bredies}
\end{equation}
has a unique solution. 
\end{prop}
\begin{proof}
Let $\tau>0$. Note that when $\|u-x\|\pnorm>1$, by Lemma \ref{normamodular2}\eqref{eq:normamaggioredi1} there holds:
$$
\rhobarpvar(u-x)\ge \frac{1}{p_+}\rhopvar (u-x)\ge \frac{1}{p_+}\|u-x\|\pnorm^{p_-}.
$$
Let now $\bar{x}\in\text{Sol(P)}$. The optimality condition reads: $0\in\nabla f(\bar{x})+\partial g(\bar{x})$ or, equivalently, $\bar{\omega}:=-\nabla f(\bar{x})\in\partial g(\bar{x})$. By definition of subdifferential, there holds $g(u)\ge g(\bar{x})+\langle \bar{\omega},u-\bar{x}\rangle=g(\bar{x})+\langle \bar{\omega},u\rangle-\langle\bar{\omega},\bar{x}\rangle$ for all $u\in\lpvar$. By combining such inequality with  the Cauchy-Schwarz inequality, we get
\begin{equation*}
\begin{aligned}
&  \rhobarpvar(u-x) + \tau \langle v^*,u\rangle+\tau g(u)\\
& \ge \frac{1}{p_+}\|u-x\|\pnorm^{p_-} + \tau \langle v^*+\bar{\omega},u\rangle+\tau g(\bar{x})-\tau\langle\bar{\omega},\bar{x}\rangle\\
& \ge \|u\|\pnorm\Big[ \frac{1}{p_+}\frac{\|u-x\|\pnorm^{p_-}}{\|u\|\pnorm} +\tau\frac{\langle v^*+\bar{\omega}, u\rangle}{\|u\|\pnorm} +\frac{\tau g(\bar{x})-\tau\langle\bar{\omega},\bar{x}\rangle}{\|u\|\pnorm}\Big]\\
& \ge \|u\|\pnorm\Big[ \frac{1}{p_+}\frac{\|u-x\|\pnorm^{p_-}}{\|u\|\pnorm} -\tau\|v^*+\bar{\omega}\|\dualnorm +\tau\frac{g(\bar{x})-\langle\bar{\omega},\bar{x}\rangle}{\|u\|\pnorm}\Big]\ge L \|u\|\pnorm
\end{aligned}
\end{equation*}
for some $L>0$ and all $u\in\lpvar$ such that $\|u\|\pnorm$ is large enough. Note, in particular, that  $\frac{1}{p_+}\frac{\|u-x\|\pnorm^{p_-}}{\|u\|\pnorm}\rightarrow +\infty$ as $\|u\|\pnorm\rightarrow+\infty$ since $p_->1$ and $x$ is fixed.
Hence, the functional in \eqref{funct_bredies} is coercive. Moreover, it is convex, proper and l.s.c. in $\lpvar$, which, by Theorem \ref{lpvar:reflexivity}, is reflexive: thus at least one solution exists. Moreover, since $1<p_-\le p_+\le2$, the functional is strictly convex, and hence the solution is unique.
\end{proof}

Clearly, the convergence analysis of Algorithm \ref{alg_bredies} is related to the study of fixed points of iterations \eqref{funct_bredies}. More precisely, we have the following result.  

\begin{prop}
The solutions of \eqref{eq:P} coincide with the fixed points of the iteration defined by Algorithm \ref{alg_bredies}.
\end{prop} 
\begin{proof}
Suppose that for some $k\geq 0$ $\xk\in\text{Sol(P)}$. Then, by optimality, $-\nabla f(\xk)\in \partial g(\xk)$. Note that $u$ solves \eqref{bredies} if and only if the following inclusion holds
\begin{equation*}
    -\tau_k\nabla f(\xk)\in\jrhobar(u-\xk)+\tau_k\partial g(u)\,,
\end{equation*}
where $\jrhobar$ is the derivative of $\rhobarpvar$, according to Proposition \ref{prop_jrho_jrhobar}. Clearly, $u=\xk$ is a solution of \eqref{bredies}, hence $\xkk=\xk$, so $\xk$ is a fixed point of the iteration process. 

Conversely, suppose now $\xk=\xkk\in\lpvar$, then
\begin{equation*}
    -\tau_k\nabla f(\xk)\in\jrhobar(\xkk-\xk)+\tau_k\partial g(\xkk)=\tau_k\partial g(\xkk)=\tau_k\partial g(\xk),
\end{equation*}
meaning that $\xk$ is optimal, that is, $\xk\in\text{Sol(P)}$.
\end{proof}

\subsection{Convergence analysis}
We provide here a detailed convergence analysis of Algorithm \ref{alg_bredies} in order to provide an insight on its convergence speed in function values. Our analysis is inspired by the one conducted in \cite{Bredies2008}, although it relies on properties related to the modular function $\rhobarpvar$ rather than to the norm $\|\cdot\|_{\lpvarb}$. As such, it relies on different arguments, as we will make precise in the following.
\begin{lemma}
For each $u \in \lpvar$, the following inequality holds true:
\begin{equation}
\label{primoboundbredies}
\left\langle\frac{\jrhobar(\xk-\xkk)}{\tau_k}, u-\xkk\right\rangle \leq  g(u)-g(\xkk)+\left\langle\nabla f(\xk), u-\xkk\right\rangle.
\end{equation}
Moreover, we denote by $D(\xk)$ the quantity 
\begin{equation}
D(\xk):=g(\xk)-g(\xkk)+\left\langle\nabla f(\xk), \xk-\xkk\right\rangle.
\label{eq:assa_D_bredies}
\end{equation}
and have that
\begin{equation}
\rhopvar(\xk-\xkk) \leq \tau_k D(\xk).
\label{eq:ass2consbredies}
\end{equation}
\label{lemma1_bredies}
\end{lemma}
\begin{proof}
Note that $\xkk$ solves \eqref{bredies} if and only if 
\begin{align}
    & 0\in\jrhobar(\xkk-\xk)+\tau_k\nabla f(\xk)+\tau_k\partial g(\xkk) \Longleftrightarrow \notag\\
    &\frac{\jrhobar(\xk-\xkk)}{\tau_k}-\nabla f(\xk) \in \partial g(\xkk).\notag
\end{align}
By definition of subdifferential, we thus have that, for all $u \in \lpvar$,
$$
\left\langle\frac{\jrhobar(\xk-\xkk)}{\tau_k}-\nabla f(\xk), u-\xkk\right\rangle \leq g(u)-g(\xkk), 
$$
which, by rearranging, coincides with \eqref{primoboundbredies}.
Choosing now $u=\xk$ above and recalling \eqref{eq:assa_D_bredies}, we get
$
\left\langle\frac{\jrhobar(\xk-\xkk)}{\tau_k}, \xk-\xkk\right\rangle \leq D(\xk).
$
Applying now \eqref{eq:Jandmodular} entails
$$
\left\langle\frac{\jrhobar(\xk-\xkk)}{\tau_k}, \xk-\xkk\right\rangle= \frac{\rhopvar(\xk-\xkk)}{\tau_k},
$$
by which \eqref{eq:ass2consbredies} follows directly.
\end{proof}
The following proposition shows that the iteration scheme \eqref{bredies} leads to a descent of the functional $\phi$ of our minimization problem \eqref{eq:P}. This will be crucial for the following convergence analysis.

\begin{prop}
For every $k\geq 0$, if $\rhopvar(\xk-\xkk)<1$, then the iteration defined by Algorithm \ref{alg_bredies} satisfies
\begin{equation}
\phi(\xkk) \leq \phi(\xk)-\Big(1-\frac{K \tau_k}{\textit{\texttt{p}}}\Big) D(\xk).
\label{eq:discesaphibredies}
\end{equation}
\label{prop:bredies_descent}
\end{prop} 
\begin{proof}
By \eqref{eq:assa_D_bredies}, we have
\begin{equation}
\label{eq:diffphibredies}
\phi(\xk)-\phi(\xkk)= D(\xk)+f(\xk)-f(\xkk)-\left\langle\nabla f(\xk), \xk-\xkk\right\rangle.
\end{equation}
For the last three terms, we notice that we can write
\begin{equation}
\begin{aligned}
&f(\xk)-f(\xkk)-\left\langle\nabla f(\xk), \xk-\xkk\right\rangle \\
&\quad=\int_{0}^{1}\left\langle\nabla f(\xk+t(\xkk-\xk))-\nabla f(\xk), \xk-\xkk\right\rangle \mathrm{d} t.
\end{aligned}
\label{eq:integralebredies}
\end{equation}
By applying now the $(\texttt{p}-1)$-H{\"o}lder continuity of $\nabla f$ (Assumption \ref{holder}), we can provide an estimate of the absolute value of the right-hand side of \eqref{eq:integralebredies}, since
\[
\begin{aligned}
&\left|\int_{0}^{1}\left\langle\nabla f(\xk+t(\xkk-\xk))-\nabla f(\xk), \xk-\xkk\right\rangle \mathrm{d} t\right| \\
&\leq \int_{0}^{1}\|\nabla f(\xk+t(\xkk-\xk))-\nabla f(\xk)\|\dualnorm \|\xk-\xkk\|\pnorm \mathrm{d} t \\
& \leq \int_{0}^{1}K\|\xk-\xkk\|\pnorm^{\textit{\texttt{p}}} t^{\textit{\texttt{p}}-1} \diff t \leq \frac{K}{\textit{\texttt{p}}}\|\xk-\xkk\|\pnorm^{\textit{\texttt{p}}} 
\end{aligned}
\]
Since we have $\rhopvar(\xk-\xkk)<1$ by construction, by Lemma \ref{normamodular2}\eqref{normamodular2_i} there holds $\|\xk-\xkk\|\pnorm<1$ and $\|\xk-\xkk\|\pnorm<\rhopvar(\xk-\xkk)^{1/p_+}\le\rhopvar(\xk-\xkk)^{1/\textit{\texttt{p}}}$ by Lemma \ref{normamodular2}\eqref{eq:normaminoredi1}. Hence, by  \eqref{eq:ass2consbredies} we obtain
$$
\begin{aligned}
&\left|\int_{0}^{1}\left\langle\nabla f(\xk+t(\xkk-\xk))-\nabla f(\xk), \xk-\xkk\right\rangle \mathrm{d} t\right| \\
& \leq \frac{K}{\textit{\texttt{p}}}\|\xk-\xkk\|\pnorm^{\textit{\texttt{p}}} 
\leq \frac{K}{\textit{\texttt{p}}}\rhopvar(\xk-\xkk) \leq \frac{K \tau_k}{\textit{\texttt{p}}} D(\xk)\,,
\end{aligned}
$$
which concludes the proof by combining this with \eqref{eq:diffphibredies} and \eqref{eq:integralebredies}.
\end{proof}
For each $k\geq 1$, let us now define for simplicity the $k$-th residual $r_{k}:=\phi(\xk)-\bar{\phi}\,.$
Note that $r_k\ge0$ by definition. We can thus rewrite \eqref{eq:discesaphibredies} as
\begin{equation}  \label{eq:residual}
r_{k}-r_{k+1} \geq\Big(1-\frac{K \tau_k}{\textit{\texttt{p}}}\Big) D(\xk).
\end{equation}
Thanks to the bounds on the step-sizes $\tau_k$, there holds $r_k-r_{k+1}\ge0$, hence the descent of the functional $\phi=f+g$ is guaranteed.

Consider now the conjugate exponent $p'(\cdot)\in\mathcal{P}(\Omega)$. Since $1<p(\cdot) \le 2$ and $\frac{1}{p(t)}+\frac{1}{p'(t)}=1\ a.e.$, there holds $2\le p'(\cdot)<+\infty$ and 
$$
(p')_-:=\essinf_{t\in\Omega} p'(t)=(p_+)' \qquad (p')_+:=\esssup_{t\in\Omega} p'(t)=(p_-)'
$$
where by $(p_+)'$ we denote the conjugate of $p_+$ and by $(p_-)'$ the conjugate of $p_-$. The following lemma shows that by assuming the boundedness of the sequence $(x_k)_k$, an estimate of the right-hand side of \eqref{eq:residual} depending on $(p_-)'$ can be found.
\begin{lemma}
If $(\xk)_k$ is bounded, then $r_k-r_{k+1}\ge c_0 r_k^{(p_-)'}$ with $c_0>0$.
\label{lemma:bredies_rate}
\end{lemma}
\begin{proof}
Since $(\xk)_k$ is bounded, there exists $C_1>0$ such that for all $k\geq 1$ $\|\xk-\bar{x}\|\pnorm \le C_{1}$ for some $\bar{x} \in \text{Sol(P)}$. The convexity of $f$ as well as \eqref{primoboundbredies} with $u=\bar{x}$ gives
\begin{align}  \label{eq:ineq1}
r_k & = f(\xk)+g(\xk)-f(\bar{x})-g(\bar{x}) \le g(\xk)-g(\bar{x})+\left\langle \nabla f(\xk), \xk-\bar{x}\right\rangle \notag \\
& =\left\langle \nabla f(\xk), \xk-\xkk\right\rangle+g(\xk)-g(\bar{x})+\left\langle \nabla f(\xk), \xkk-\bar{x}\right\rangle\notag  \\
& =D(\xk)+g(\xkk)-g(\bar{x})+\left\langle \nabla f(\xk), \xkk-\bar{x}\right\rangle \notag \\
& \le D(\xk)+\left\langle\frac{\jrhobar(\xk-\xkk)}{\tau_k}, \xkk-\bar{x}\right\rangle \notag \\
& \le D(\xk)+\tau_k^{-1}\|\jrhobar(\xk-\xkk)\|\dualnorm \|\xkk-\bar{x}\|\pnorm \notag  \\
& \le D(\xk)+\tau_k^{-1}\|\jrhobar(\xk-\xkk)\|\dualnorm C_1.
\end{align}

Recalling now that $\langle\jrhobar(u),v\rangle=\int_\Omega \sign(u(t))|u(t)|^{p(t)-1}v(t)\diff t$ for any $u,v\in\lpvar$,
by\eqref{eq:isom_ass} and Proposition \ref{prop:normaspassduale} we have
$$\|\jrhobar(u)\|\dualnorm=\|\sign(u)|u|^{p(\cdot)-1}\|\spassnorm\le2\|\sign(u)|u|^{p(\cdot)-1}\|\qnorm.$$
Observe now that the following equality holds true:
$$
\begin{aligned}
\rhoqvar(\sign(u)|u|^{p(\cdot)-1})&=\int_\Omega \Big(\sign(u(t))|u(t)|^{p(t)-1}\Big)^{p'(t)}\diff t\\
&= \int_\Omega \Big(\sign(u(t))|u(t)|\Big)^{p(t)}=\int_\Omega |u(t)|^{p(t)}\diff t=\rhopvar(u).
\end{aligned}
$$
Hence it follows that 
$
\rhoqvar\Big(\sign(\xk-\xkk)|\xk-\xkk|^{p(\cdot)-1}\Big)<1,
$ since by construction we have $\rhopvar(\xk-\xkk)<1$. 
Together with Lemma \ref{normamodular2}\eqref{normamodular2_i}, this leads to
$
\|\sign(\xk-\xkk)|\xk-\xkk|^{p(\cdot)-1}\|\qnorm<1.
$
Furthermore note that for all $u\in\lpvar$, keeping in mind that $(p')_+=(p_-)'$, Lemma \ref{normamodular2}\eqref{eq:normaminoredi1} entails
$$
\|\sign(u)|u|^{p(\cdot)-1}\|\qnorm\le\Big(\rhoqvar(\sign(u)|u|^{p(\cdot)-1})\Big)^{1/(p_-)'},
$$
which can now be evaluated in $u=\xk-\xkk$ and combined with the previous inequalities to get from \eqref{eq:ineq1}
$$
\begin{aligned}
r_k & \le D(\xk)+\tau_k^{-1}\|\jrhobar(\xk-\xkk)\|\dualnorm C_1\\
& \le D(\xk)+\tau_k^{-1} C_1 2\|\sign(\xk-\xkk)|\xk-\xkk|^{p(\cdot)-1}\|\qnorm\\
& \le D(\xk)+\tau_k^{-1} C_1 2\Big(\rhoqvar(\sign(\xk-\xkk)|\xk-\xkk|^{p(\cdot)-1})\Big)^{1/(p_-)'}\\
& = D(\xk)+\tau_k^{-1} C_1 2\Big(\rhopvar(\xk-\xkk)\Big)^{1/(p_-)'},\\
\end{aligned}
$$
so that, by \eqref{eq:ass2consbredies}, we have 
$r_k \le D(\xk)+2 C_1 \tau_k^{-1} \left(\tau_k D(\xk) \right)^{1/(p_-)'}.$
The step-size constraints \eqref{eq:step_bunds_algbredies} together with \eqref{eq:discesaphibredies} entail $r_k-r_{k+1}\ge \delta D(x^k)$ and $\tau_k\ge \bar{\tau}$. Plugging these quantities into the last inequality above, we obtain
\begin{equation*}
\begin{aligned}
r_k & \le \frac{r_k-r_{k+1}}{\delta}+2 C_1\bar{\tau}^{-\frac{(p_-)'-1}{(p_-)'}}\left(\frac{r_k-r_{k+1}}{\delta}\right)^{1/(p_-)'}.
\end{aligned}
\end{equation*}
Note that since $r_k$ is a nonnegative decreasing sequence, then $r_k-r_{k+1}$ is bounded, so that we can write
\[
\delta r_{k} \le (r_k-r_{k+1})^{1/(p_-)'}\left( R^{-\frac{(p_-)'-1}{(p_-)'}} + 2 C_1\bar{\tau}^{-\frac{(p_-)'-1}{(p_-)'}}\delta^{\frac{(p_-)'-1}{(p_-)'}}\right)\,,
\]
for a sufficiently large $R>0$, which finally gives
\[
r_k-r_{k+1} \ge \frac{\delta^{(p_-)'}}{\left( R^{-\frac{(p_-)'-1}{(p_-)'}} + 2 C_1\bar{\tau}^{-\frac{(p_-)'-1}{(p_-)'}}\delta^{\frac{(p_-)'-1}{(p_-)'}}\right)^{(p_-)'}} r_k^{(p_-)'}.
\]
\end{proof}

Thanks to the previous lemma and following analogous arguments as in \cite[Proposition 4]{Bredies2008}, we obtain the following convergence result.

\begin{prop}
\label{prop:convergenceratebredies}
If $(\xk)_k$ is bounded, then the following convergence rate in function values can be found for the iterates of Algorithm \ref{alg_bredies} 
\begin{equation}
r_k \le \eta \frac{1}{k^{p_- -1}}.
\label{eq:conv_Bredies} 
\end{equation}
\end{prop}


Such convergence rate is related to the smoothness of the space $\lpvarb$ considered and, in particular, to the infimum exponent value $p_-$ appearing in the analytical proof of Lemma \ref{lemma:bredies_rate}. It can thus be read as the worst-case convergence speed. It is highly expected that such convergence result can be improved and that, practically, faster convergence could be achieved, as our numerical tests will show.  

We now provide a result relative to the convergence of the sequence $(x_k)_k$ itself.
\begin{prop}
	If $(\xk)_k$ is bounded, then the sequence $(\xk)_k$ has at least one accumulation point. All accumulation points belong to $\text{Sol(P)}$. If $\text{Sol(P)}=\{\bar{x}\}$, then $(\xk)_k$ converges weakly to $\bar{x}$.
\end{prop}

\begin{proof}
	See \cite[Proposition 5]{Bredies2008} and \cite[Proposition 3.4 (iii)]{GuanSong2015}.
\end{proof}

We conclude this section recalling  the definition of totally convex and $r$-convex functions. Under this further hypothesis, it is possible to show that the sequence of iterates defined by Algorithm \ref{alg_bredies} converges strongly to a solution of the minimization problem \eqref{eq:P} and, moreover, that \eqref{eq:P} has a unique solution.
\begin{definition}
Let $h:\lpvar \longrightarrow \R\cup\{+\infty\}$ be proper, convex and l.s.c. The functional $h$ is said to be totally convex in $\hat{u}\in\lpvar$ if  for all $\omega\in\partial h(\hat{u})$ and for each $(u^n)_n$ such that 
\begin{equation*}
h(u^n)-h(\hat{u})-\langle \omega, u^n-\hat{u}\rangle \rightarrow0,
\end{equation*}
there holds
$
\|u^n-\hat{u}\|\pnorm\rightarrow0,
$
for $n\to+\infty$. We say that $h$ is totally convex if it is totally convex in $\hat{u}$ for all $\hat{u}\in\lpvar$.

Similarly, $h$ is convex of power-type $r$ (or $r$-convex) in $\hat{u}\in\lpvar$ with $r\ge2$ if for all $M>0$ and $\omega\in\partial h(\hat{u})$ there exists $\beta>0$ such that for all $\|u-\hat{u}\|\pnorm\le M$ 
\begin{equation*}
h(u)-h(\hat{u})-\langle \omega, u-\hat{u}\rangle \ge \beta \|u-\hat{u}\|\pnorm^r.
\end{equation*}
We say that $h$ is convex of power-type $r$ if it is convex of power-type $r$ in $\hat{u}$ for all $\hat{u}\in\lpvar$.
\end{definition}
\begin{prop}
If $f$ or $g$ is totally convex or $r$-convex with $r\ge2$, the solution of problem \eqref{eq:P} is unique. Denoting it by $\bar{x}\in\lpvar$, we further have that under the same conditions the sequence $(\xk)_k$ defined by \eqref{bredies} converges strongly to $\bar{x}$.
\end{prop}
\begin{proof}
First, note that $r$-convexity implies total convexity. Then, focus on the case that $f$ is totally convex. 
Suppose now by contradiction that there exists $\tilde{x}\in\text{Sol(P)}$ with $\tilde{x}\not=\bar{x}$. Thus $\|\tilde{x}-\bar{x}\|\pnorm>0$. By defining
\begin{equation}  \label{eq:subdiff_R}
R(x):=f(x)-f(\bar{x})-\langle \nabla f(\bar{x}), x-\bar{x}\rangle,
\end{equation}
we have that $R(x)\ge0$ for all $x\in\lpvar$. Moreover, by the optimality of $\bar{x}$ and the subgradient inequality, there holds  $\phi(x)- \phi(\bar{x})\ge R(x)$ for all $x\in\lpvar$. Choosing $x=\tilde{x}$ , we thus get 
$
0=\phi(\tilde{x})-\phi(\bar{x})\ge R(\tilde{x})\ge0,
$
whence $R(\tilde{x})=0$. Taking now $u^n=\tilde{x}$ for all $n\geq 1$, we find a contradiction as the total convexity property is violated, whence we deduce $\tilde{x}=\bar{x}$.

To complete the proof, let us consider \eqref{eq:subdiff_R} once again. By optimality of $\bar{x}$ and thanks to the subgradient inequality, there holds $r_k\ge R(\xk)$, whence, by Proposition \ref{prop:convergenceratebredies}, we deduce that $R(\xk)\le \eta\frac{1}{k^{p_--1}}$. Letting now $k\to+\infty$ we thus infer that $R(\xk)\rightarrow 0$ and, by the total convexity of $f$, that $\|\xk-\bar{x}\|\pnorm\rightarrow0$, which completes the proof.

In the case that $g$ is totally convex, the proof is analogous by defining $R(x):=g(x)-g(\bar{x})+\langle \nabla f(\bar{x}), x-\bar{x}\rangle$.
\end{proof}

\section{A Bregmanized modular-proximal gradient algorithm}
\label{sec:Guan_Song_lpvar}

In this section, we introduce a different modular-proximal gradient algorithm solving \eqref{eq:P} where the proximal step is defined in terms of a modular Bregman-like distance. Our study is here inspired by the analysis carried out in \cite{GuanSong2015} where an analogous algorithm is studied for a general Banach space $\mathcal{X}$. Similarly as above, we start this section by stating the required assumptions.  
\begin{ass}
\label{ass1}
$\nabla f : \lpvar \longrightarrow \lpvar^*$ is $(\textit{\texttt{p}}-1)$H{\"o}lder-continuous with $1<\textit{\texttt{p}}\le2$ with constant $K$. 
\end{ass}
\begin{ass}
\label{ass2}
There exists $c>0$ such that for all $u,v\in\lpvar$
\[
\langle \jrhobar(u)-\jrhobar(v), u-v\rangle \ge c \max \left\{ \|u-v\|\pnorm^\textit{\texttt{p}}, \|\jrhobar(u)-\jrhobar(v)\|\dualnorm^{\textit{\texttt{p}}'}\right\} 
\]
where $\textit{\texttt{p}}'$ is the H\"{o}lder-conjugate of $\textit{\texttt{p}}$, i.e. $\frac{1}{\textit{\texttt{p}}}+\frac{1}{\textit{\texttt{p}}'}=1$.
\end{ass}
\begin{algorithm}[h!]
\caption{Bregmanized modular-proximal gradient algorithm in $\lpvar$ spaces}
\label{alg_guansong}
\small{
\textbf{Parameters:} $\{\tau_k\}_k$ s.t. \begin{equation}
0<\Bar{\tau}\le \tau_k \le \frac{\textit{\texttt{p}}c(1-\delta)}{K} \qquad \text{with} \ 0<\delta<1.
\end{equation}
\textbf{Initialization:} Start with $x^0\in\lpvar$.\\
{
\textsc{FOR $k=0,1,\ldots$ REPEAT}}
\begin{itemize}
        \item[] Compute the next iterate as: 
        \begin{equation}
    \xkk = \argmin_{u\in\lpvar} \rhobarpvar(u)-\langle \jrhobar(\xk), u\rangle + \tau_k \langle \nabla f(\xk),u\rangle+\tau_k g(u)
    \label{alg1}
\end{equation}
    \end{itemize}
{\textsc{UNTIL}} convergence
}
\end{algorithm}

The pseudocode of the proposed algorithm is reported in Algorithm \ref{alg_guansong}. 

Assumption \ref{ass2} links the geometrical properties of the space $\lpvar$ with the H{\"o}lder smoothness properties of $f$. It has to be interpreted as a sufficient compatibility condition between the ambient space $\lpvar$ and the function $f$ for achieving the desired convergence result. We will comment more on the practical verifiability of this condition in the following sections.

The optimization problem \eqref{alg_guansong} characterizing Algorithm \ref{alg_guansong} can be linked to suitably defined Bregman distances, see, e.g., \cite{burger2015bregman,benning2021} for more details. We recall in the following their definition.

\begin{definition}
Given $h:\lpvar\longrightarrow \R\cup\{+\infty\}$ smooth, convex, proper, and l.s.c., the Bregman distance between $x,y\in\lpvar$ with respect to $h$ is defined by 
\[
\mathcal{D}_h(x,y):=h(x)-h(y)-\langle \nabla h(y),x-y\rangle \qquad \forall x,y\in\lpvar.
\]
\end{definition}
We now observe that the updating rule in Algorithm \ref{alg_guansong} can be equivalently written in terms of the Bregman distance  associated to $\rhobarpvar$ as follows:
\begin{equation}
\xkk \in \argmin_{u\in\lpvar} \dbreg(u,\xk) + \tau_k \langle \nabla f(\xk),u\rangle+\tau_k g(u),
\label{alg2}
\end{equation}
since the terms which are constant with respect to $u$ can trivially be neglected.

We report now the analogue of Proposition \ref{prop:bredies1}, omitting the proof since the reasoning is similar as to the proof of Proposition \ref{prop:bredies1}. It shows that the minimization problem \eqref{alg1} has a unique solution at each iteration, and thus $\xkk$ is well defined.
\begin{prop}
The problem 
\begin{equation*}
\argmin_{u\in\lpvar} \rhobarpvar(u)-\langle \jrhobar(x), u\rangle + \tau \langle v^*,u\rangle+\tau g(u)
\end{equation*}
has a unique solution for each $x\in\lpvar$, $v^*\in\lpvar^*$ and $\tau>0$.
\end{prop}

In order to interpret \eqref{alg1} as a fixed-point iterative scheme, it is useful to introduce the following notion which shows analogies to the standard scheme of the Moreau envelope (see, e.g., \cite[Chapter 12]{BauschkeCombettes2017}).

\begin{definition}
Given $h:\lpvar\longrightarrow \R\cup\{+\infty\}$ smooth, convex, proper, l.s.c., we define the Moreau-like envelope $e_h:\lpvar^*\longrightarrow\R$ and the modular-proximal mapping $\pi_h:\lpvar^*\longrightarrow\lpvar$ as follows
\begin{align}
&e_h(x^*):=\inf_{u\in\lpvar}{h(u)+\Delta(x^*,u)}, \quad x^*\in\lpvar^* \notag\\
&\pi_h(x^*):=\argmin_{u\in\lpvar}{h(u)+\Delta(x^*,u)}, \quad x^*\in\lpvar^*
\label{eq:prox}
\end{align}
where $\Delta (\cdot,\cdot)$ denotes the Bregman-like distance associated to $\rhobarpvar$, that is $\Delta (x^*,u):= \rhobarpvar(u)-\langle x^*,u\rangle$.
\end{definition}
Note that the minimum in \eqref{eq:prox} is uniquely attained as $\rhobarpvar$ is a strictly convex function. Moreover, the unique point $\pi_h(x^*)$ satisfies 
\begin{equation}
0\in\partial h(\pi_h(x^*))+\jrhobar(\pi_h(x^*))-x^* \iff x^*\in \partial h(\pi_h(x^*))+\jrhobar(\pi_h(x^*)).
\label{eq:dopodefprox}
\end{equation}

\begin{prop}
For any $\gamma>0$, there holds $\bar{x}\in\text{Sol(P)}$ if and only if $\bar{x}=\pi_{\gamma g}\big(\jrhobar(\bar{x})-\gamma \nabla f(\bar{x})\big)$.
\end{prop}
\begin{proof}
Since $\bar{x}$ solves \eqref{eq:P}, we have
$$
0 \in \nabla f(\bar{x}) + \partial g (\bar{x}) \quad \iff \quad 
0  \in \gamma \nabla f(\bar{x}) - \jrhobar(\bar{x}) +\jrhobar(\bar{x}) + \gamma \partial g (\bar{x})\,,$$
that is
$\jrhobar(\bar{x})-\gamma \nabla f(\bar{x})  \in \jrhobar(\bar{x}) + \gamma \partial g (\bar{x})\,.$
By \eqref{eq:dopodefprox}, we thus deduce as required that
$\bar{x}=\pi_{\gamma g}\big(\jrhobar(\bar{x})-\gamma \nabla f(\bar{x})\big)$.
\end{proof}

\begin{oss} \label{oss}
The iteration \eqref{alg1} can be equivalently formulated as 
\begin{equation*}
\begin{aligned}
\xkk & = \jrhobar(\xk)-\tau_k\nabla f(\xk)\in \jrhobar(\xkk)+\tau_k\partial g(\xkk) \iff\\
\xkk & = \pi_{\tau_k g}\big( \jrhobar(\xk)-\tau_k\nabla f(\xk)\big)\,.
\end{aligned}
\end{equation*}
This shows that  \eqref{alg1} can be read as a fixed-point iteration scheme. 
\end{oss}

Similarly as in Proposition \ref{prop:convergenceratebredies}, it is possible to prove the convergence in function values for the iterates of  Algorithm \ref{alg_guansong} and achieve a convergence rate. Here, we omit the proof, as it follows verbatim the one in \cite{GuanSong2015} which is itself inspired by \cite{Bredies2008}.
\begin{prop}
\label{prop:convergencerate}
If $(\xk)_k$ is bounded, then the following convergence rate in function values can be found for the iterates of Algorithm \ref{alg_guansong} 
\begin{equation}
\label{eq:conv_GS}
r_k \le \eta \frac{1}{k^{\textit{\texttt{p}}-1}}.
\end{equation}
\end{prop}


It is interesting to compare the rate \eqref{eq:conv_GS} with the analogous one in \eqref{eq:conv_Bredies} obtained for Algorithm \ref{alg_bredies}. The dependence on the H\"{o}lder exponent $\texttt{p}$ in \eqref{eq:conv_GS} links the speed of convergence to the smoothness of the smooth function $f$ rather than to the one of the underlying $\lpvar$ space, which appears in \eqref{eq:conv_Bredies}. For reasonably smooth problems we thus expect Algorithm \ref{alg_guansong} to show better performance than Algorithm \ref{alg_bredies}.
\begin{oss}
Algorithm \ref{alg_guansong} can be equivalently formulated in terms of the modular function $\rhopvar$ maintaining the same convergence rate and convergence analysis. However, $\rhobarpvar$ intuitively better generalizes the norm $\frac{1}{\texttt{p}}\|\cdot\|_\mathcal{X}$ used in the definition of Guan and Song's algorithm \cite{GuanSong2015}. We underline, on the other hand, that Algorithm \ref{alg_bredies} has to be defined in terms of $\rhobarpvar$, since the convergence analysis cannot be carried out otherwise. In particular, in the proof of Lemma \ref{lemma1_bredies}, \eqref{eq:Jandmodular} is necessary and it requires the use of $\rhobarpvar$. 
\end{oss}
\section{Sparse reconstruction models: thresholding functions and primal-dual interpretation}
\label{sec:Applications}

In this section, we consider an exemplar sparse reconstruction model used in a variety of signal/image inverse problems and discuss the application of the algorithms presented in this paper for the computation of its numerical solution. Given a Lebesgue measurable map $p(\cdot):\Omega\longrightarrow(1,2]$, we consider a bounded linear operator $A:\lpvar\longrightarrow L^{p_+}(\Omega)$ and an observation $y\in L^{p_+}(\Omega), ~p_+\leq 2$. For $\lambda>0$, we aim to minimize the Tikhonov-like functional
\begin{equation*}
    \argmin_{x\in\lpvar}~ \frac{1}{p_+}\|Ax-y\|_{p_+}^{p_+} +\lambda \|x\|_1
    \label{var_model_1}
\end{equation*}
in the Banach space $\lpvar$, where $f(x)=\frac{1}{p_+}\|Ax-y\|_{p_+}^{p_+}$ is proper, convex and smooth, while $g(x)=\lambda \|x\|_1$ is proper, l.s.c, convex and nonsmooth. 

The gradient of $f$ can be computed as $\nabla f(x)= A^* \Jppiuppiu (Ax-y) \in \lpvardual.
$
In agreement with Assumption \ref{holder} for Algorithm \ref{alg_bredies} and with Assumption \ref{ass1} for Algorithm \ref{alg_guansong}, we now need to study the H{\"o}lder continuity of $\nabla f$. We claim that $\nabla f$ is $(p_+ -1)$-H\"{o}lder continuous. To show this, we recall the following useful definitions and properties.

\begin{definition} \cite{Schuster2012}
A Banach space $\mathcal{X}$ is called smooth of power-type $r$ (or $r$-smooth) with $r \in ( 1,2]$ if there exists a constant $C>0$ such that for all $u, v \in \mathcal{X}$ it holds that
$$
\frac{\|v\|_{\mathcal{X}}^{r}}{r}-\frac{\|u\|_{\mathcal{X}}^{r}}{r}-\langle \mathbf{J}_\mathcal{X}^{r}(u), v-u\rangle \le C\|v-u\|_{\mathcal{X}}^{r},
$$
where $\mathbf{J}_\mathcal{X}^{r}(\cdot)$ denotes the duality mapping between $\mathcal{X}$ and $\mathcal{X}^{*}$.
\end{definition}
\begin{prop} \cite{Bredies2008}
If a Banach space $\mathcal{X}$ is smooth of power-type $r$, then $\frac{\|\cdot\|_{\mathcal{X}}^{r}}{r}$ is continuously differentiable with derivative $\mathbf{J}_\mathcal{X}^{r}$, which is  $(r-1)$-H{\"o}lder continuous.

Furthermore, for $s \ge r$, the functional $\frac{\|\cdot\|_{\mathcal{X}}^{s}}{s}$ is continuously differentiable. Its derivative is given by $\mathbf{J}_\mathcal{Y}^{s}$, which is still $(r-1)$-H{\"o}lder continuous on each bounded subset of $\mathcal{X}$.
\end{prop}
\begin{lemma} \cite{Schuster2012} 
Constant exponent Lebesgue spaces $L^r(\Omega)$ 
are $\min\{2,r\}$-smooth.
\end{lemma}
The space $L^{p_+}(\Omega)$ is thus $p_+$-smooth and $\Jppiuppiu$ is $(p_+-1)$-H{\"o}lder continuous, i.e.
\[
\exists K_1>0 \quad\text{s.t.}\quad \forall y_1,y_2\in L^{p_+}(\Omega)\quad  \|\Jppiuppiu(y_1)-\Jppiuppiu(y_2)\|_{({p_+})^*}\le K_1 \|y_1-y_2\|_{p_+}^{p_+-1}.
\]
Using this combined with the linearity of $A$ and the submultiplicativity of the norm, we can thus write
\begin{equation*}
\begin{aligned}
& \|\nabla f(u)-\nabla f(v)\|\dualnorm = \Big\| A^*\Big[ \Jppiuppiu (Au-y) - \Jppiuppiu (Av-y) \Big] \Big\|\dualnorm \\ 
& \le \|A^*\|\dualnorm \|\Jppiuppiu (Au-y) - \Jppiuppiu (Av-y)\|\dualnorm\le \|A^*\|\dualnorm K_1 \|A(u-v)\|_{p_+}^{p_+-1}\\
& \le K_1 \|A^*\|\dualnorm \|A\|_{p_+}^{p_+-1}\|u-v\|\pnorm^{p_+ -1} \le K\|u-v\|\pnorm^{p_+ -1} \qquad \forall \ u,v\in\lpvar,
\end{aligned}
\end{equation*}
showing that $\nabla f$ is $(p_+-1)$-H{\"o}lder continuous, as required. 

We can thus focus now on the computation of the solutions of \eqref{bredies} in the discrete setting, where the domain $\Omega$ is discretized into the disjoint sum of $n$ nonempty subsets, i.e. $\Omega = \uplus_{i=1}^n \Omega_i$. By considering a single real value on each subset $\Omega_i$, with a slight abuse of notation we simply denote by $\ell^{p(\cdot)}(\mathbb{R}^n)$, the $n$-th dimensional subspace of the sequence space $\ell^{p(\cdot)}$ generated by the first $n$ elements $e_1,e_2,\ldots,e_n$ of the canonical basis. To allow effective numerical resolution, we heavily exploit the separability property of the operators involved in the sense of Definition \ref{def:add_sep}.
By setting $\sigma^k:=A^* \Jppiuppiu (A\xk-y) \in \mathbb{R}^n$, the iteration \eqref{bredies}  of Algorithm \ref{alg_bredies} reads as
\begin{align}
 \xkk  & = \argmin_{u\in\ell^{p(\cdot)}(\mathbb{R}^n)} \left\{ \rhobarpvar(u-\xk) + \tau_k \langle \sigma^k,u\rangle+\tau_k \lambda \|u\|_1\right\}   \label{sparse1}\\
& = \argmin_{u\in\ell^{p(\cdot)}(\mathbb{R}^n)}  \sum_{i=1}^n \left\{ \frac{1}{p_i}|u_i-\xk_i|^{p_i}+\tau_k \sigma^k_i u_i + \tau_k \lambda |u_i| \right\}. \label{sparse2}
\end{align} 
Note that, thanks to the additive separability property, at each $k$-th iteration, with $k\geq 1$,  the $n$-dimensional minimization problem in \eqref{sparse1} corresponds to the sum of $n$ 1D problems. Hence, each component can be treated independently, so one can consider the independent minimization of the 1D functions:
\[
\Psi^{\text{Alg.P}}_{x,s,t,p}(u):=\frac{1}{p}|u-x|^p+s u + t|u|,
\]
with $x=\xk_i$, $s=\tau_k\sigma^k_i$, $t=\tau_k \lambda$ and $p=p_i$, where Alg.P stands for primal algorithm, which will be better clarified in Section \ref{sec:interpretation}. The minimizers of $\Psi^{\text{Alg.P}}_{x,s,t,p}$ can be computed by optimality as $ \Big(\partial\Psi^{\text{Alg.P}}_{x,s,t,p}\Big)^{-1}(0)$ and can be expressed in a compact form in terms of the thresholding function (see \cite{Bredies2008} for more details): 
\begin{equation}
\label{eq:thresholding1}
    T^{\text{Alg.P}}(x,s,t,p)=
    \begin{cases}
    x-\sign(s+t)|s+t|^{\frac{1}{p-1}} & \text{if } \, x > \sign(s+t)|s+t|^{\frac{1}{p-1}}\\
    x-\sign(s-t)|s-t|^{\frac{1}{p-1}} & \text{if } \, x < \sign(s-t)|s-t|^{\frac{1}{p-1}}\\    0 & \text{otherwise,}
    \end{cases}
\end{equation} 
\begin{figure}[t!]
\centering
\begin{subfigure}[b]{0.325\textwidth}
\includegraphics[width=\textwidth]{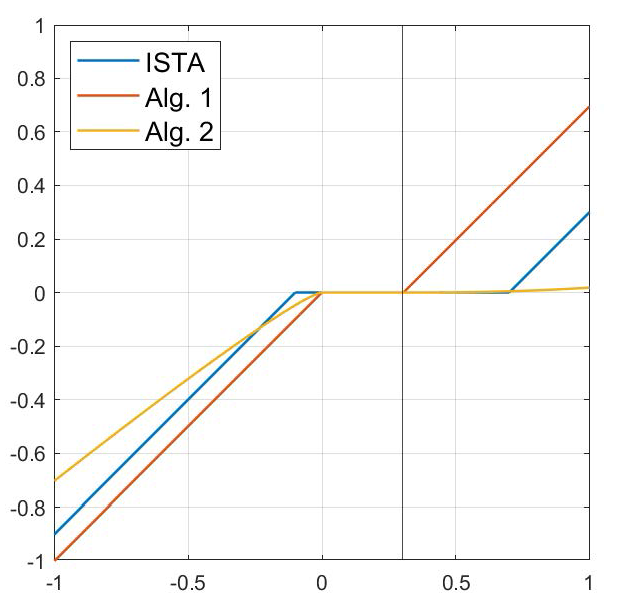}
\caption{Thresh. functions}
\label{fig:threshold(a)}
\end{subfigure}
\begin{subfigure}[b]{0.325\textwidth}
\includegraphics[width=\textwidth]{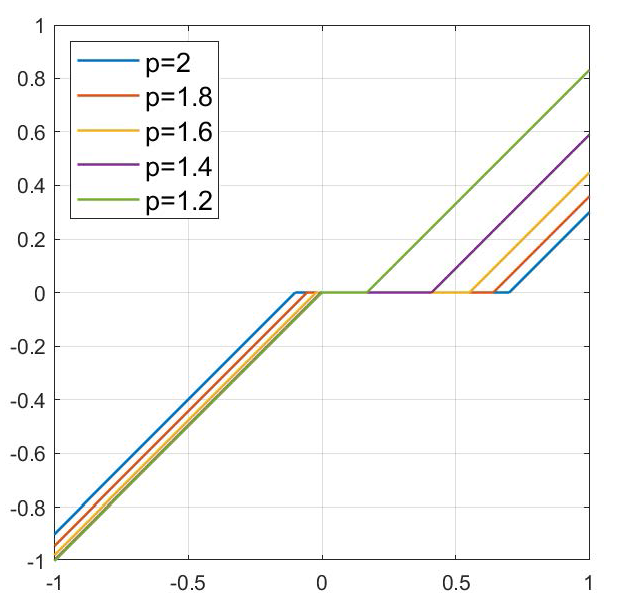}
\caption{$T^{\text{Alg.P}}(\cdot,s,t,p)$}
\label{fig:threshold(b)}
\end{subfigure}
\begin{subfigure}[b]{0.325\textwidth}
\includegraphics[width=\textwidth]{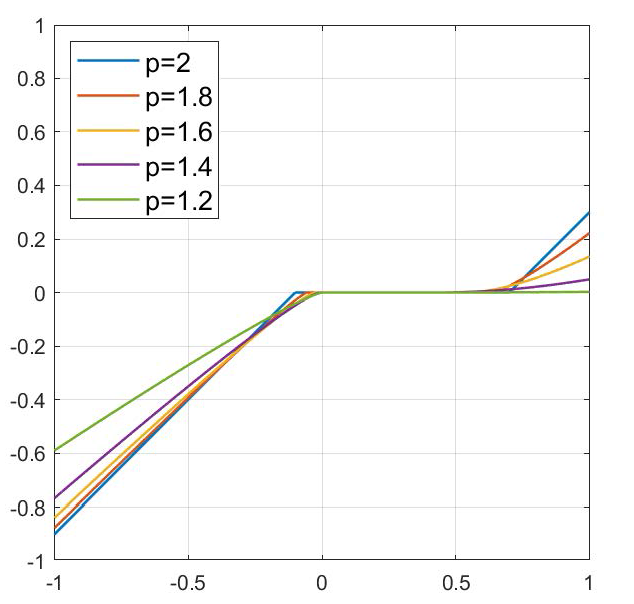}
\caption{$T^{\text{Alg.D}}(\cdot,s,t,p)$}
\label{fig:threshold(c)}
\end{subfigure}
\caption{(a) 1D thresholding functions $T^{\text{Alg.P}}(\cdot,s,t,p)$, $T^{\text{Alg.D}}(\cdot,s,t,p)$ and $T^{\text{ISTA}}(\cdot,s,t)$ with $p=1.3$, $s=0.3$, $t=0.4$. (b) $T^{\text{Alg.P}}(\cdot,s,t,p)$ with $s=0.3$, $t=0.4$ and $p\in\left\{1.2,1.4,1.6,1.8,2\right\}$. (c) $T^{\text{Alg.D}}(\cdot,s,t,p)$ with $s=0.3$, $t=0.4$ and $p\in\left\{1.2,1.4,1.6,1.8,2\right\}$.
}
\label{fig:threshold}
\end{figure}
As discussed in the previous sections, a direct application of the proximal-gradient algorithm in Banach spaces studied in \cite{Bredies2008} to the $\ell^{p(\cdot)}(\mathbb{R}^n)$ scenario would have been rendered more challenging due to the nonseparability of the norm $\|\cdot\|_{\ell^{p(\cdot)}}$ and of the corresponding $r$-duality map, given by
\begin{equation*}
    \mathbf{J}_{\ell^{p(\cdot)}(\mathbb{R}^n)}^r(x)= \frac{1}{
 {\sum_{j=1}^n\frac{p_j |x_j|^{p_j}}{\|x\|_{\ell^{p(\cdot)}}^{p_j}}}
 }
 \Big({\frac{1}{\|x\|_{\ell^{p(\cdot)}} ^{{p_i}-{r}}}{p_i} |x_i|^{{p_i}-1}\sign(x_i)\Big)_{i=1,\cdots,n}}.
\end{equation*}

We can similarly focus on Algorithm \ref{alg_guansong}, about which, we remark, Assumption \ref{ass1} holds with $\texttt{p}=p_+$. Assumption \ref{ass2} is hard to verify in practice, although numerical tests show that it is not that challenging to find a step-size for which convergence is guaranteed. Proceeding similarly as above, we exploit again the separability of the modular appearing in the computation of the $k$-th iteration \eqref{alg1} of Algorithm \ref{alg_guansong}, which leads to the computation of the minimizers of the following $1$D function 
$$
\Psi^{\text{Alg.D}}_{x,s,t,p}(u)=\frac{1}{p}|u|^p-{|x|^{p-1}\sign(x)}u+{s} u + t|u|
$$
where $x=(\xk)_i,\ s=\tau_k\sigma^k_i,\ t=\tau_k \lambda,\ p=p_i$ and Alg.D stands for dual algorithm as explained in Section \ref{sec:interpretation}. Such minimizers are given by $\Big(\partial\Psi^{\text{Alg.D}}_{x,s,t,p}\Big)^{-1}(0)$ and correspond to the following thresholding function:
\begin{equation} 
\label{eq:thresholding2} 
    T^{\text{Alg.D}}(x,s,t,p)=
    \begin{cases}
    (|x|^{p-1}\sign(x)-s-t)^{\frac{1}{p-1}} & \text{if } \, |x|^{p-1}\sign(x)>s+t\\
    -(s-t-|x|^{p-1}\sign(x))^{\frac{1}{p-1}} & \text{if } \, |x|^{p-1}\sign(x)<s-t\\
    0 & \text{otherwise.}
    \end{cases}
\end{equation}
Figure \ref{fig:threshold(a)} shows a comparison between the thresholding functions \eqref{eq:thresholding1} and \eqref{eq:thresholding2} with the classical soft-thresholding function 
\begin{equation}   \label{eq:thresholdingISTA}
    T^{\text{ISTA}}(x,s,t)=
    \begin{cases}
    x-s-t & \text{if } \, x>s+t\\
    x-s+t & \text{if } \, x<s-t\\
    0 & \text{otherwise,}
    \end{cases}
\end{equation}
which corresponds to the minimization of the $1$D function $\Psi^{\text{ISTA}}_{x,s,t}(u)=\frac{1}{2}(u-x+s)^2+t|u|$ appearing in the solution of the $\ell_2-\ell_1$ LASSO optimization problem in the Hilbert space $\ell_2$ by means of the standard ISTA algorithm \cite{Daubechies2003}.
It is interesting to point out that differently from the ISTA thresholding function \eqref{eq:thresholdingISTA}, both  thresholding functions \eqref{eq:thresholding1} and \eqref{eq:thresholding2}  are no longer symmetrical with respect to the vertical line $x=s$. Moreover, 
$T^{\text{Alg.P}}(\cdot,s,t,p)$ remains linear in $x$ similarly to $T^\text{ISTA}(\cdot,s,t)$, as shown in Figure \ref{fig:threshold(b)} for some exemplar values of $p$, while the thresholding function $T^{\text{Alg.D}}$ is, in general, nonlinear, as shown in Figure \ref{fig:threshold}. Note that for $p=2$, both $T^{\text{Alg.P}}(\cdot,s,t,2)$ and $T^{\text{Alg.D}}(\cdot,s,t,2)$ coincide with the standard soft-thresholding $T^{\text{ISTA}}(\cdot,s,t)$ operator.

\subsection{Interpretation as primal and dual algorithms}  \label{sec:interpretation}
In this section, we analyze the relationship between the thresholding functions \eqref{eq:thresholding1} and \eqref{eq:thresholding2}  of algorithms \ref{alg_bredies} and \ref{alg_guansong}, respectively, and the primal and dual iterative methods described, e.g., in \cite{Schuster2012}, similarly to the analogous considerations sketched in Section \ref{sec:intro} for the algorithms in \cite{Bredies2008} and  \cite{GuanSong2015} in the case $g\equiv 0$. In order to extend our considerations to the nonsmooth case, we need to provide an analytical expression of duality mappings in $\lp$ spaces, which, for better readability, will be done in the discrete setting, that is in $\ell^{p(\cdot)}(\mathbb{R}^n)$.

As a first remark, notice that given a fixed $p$ and $r$ such that $1<p,r<+\infty$, the $r$-duality mapping of $\ell^{p}(\mathbb{R}^n)$ can be defined in terms of componentwise maps $\mathbf{j}_p:\mathbb{R}\to\mathbb{R}$ defined by
$
         \mathbf{j}_{p}(t) := |t|^{p-1}\sign(t)\,,
$
as follows
$$ \;
\mathbf{J}_{\ell^p}^{r}(x)=
        \|x\|_{p}^{r-p}\left(
        \mathbf{j}_{p}(x_i)
        \right)_{i=1,\ldots,n}.
$$
Thanks to the isometric isomorphism between $\left(\ell^p\right)^*$ and $\ell^{p'}$, the inverse of such duality mapping is $\left(\mathbf{J}_{\ell^p}^{r}\right)^{-1}=\mathbf{J}_{(\ell^p)^*}^{r'}=\mathbf{J}_{\ell^{p'}}^{r'}$, which can be explicitly written as:
\begin{align}
    \mathbf{J}_{\ell^{p'}}^{r'}(x)& =
     \|x\|_{p'}^{r'-p'}
    \left(
        \mathbf{j}_{p'}(x_i)
        \right)_{i=1,\ldots,n} =
    \|x\|_{p'}^{r'-p'}\Big(|x_i|^{p'-1}\sign(x_i)\Big)_{i=1,\cdots,n} \notag\\
    & = \|x\|_{p'}^{r'-p'}\Big(|x_i|^{\frac{1}{p-1}}\sign(x_i)\Big)_{i=1,\cdots,n}. \notag
\end{align}
Comparing now the thresholding function $T^{\text{Alg.P}}(\cdot,s,t,p)$ in \eqref{eq:thresholding1} with the above formulas, we observe that we can rewrite it in terms of the componentwise maps $\mathbf{j}_{p'}$ associated to the duality map $\mathbf{J}_{\ell^{p'}}^{r'}$ as:
\begin{equation}\label{eq:thresholding1_new}
    T^{\text{Alg.P}}(x,s,t,p)=
    \begin{cases}
    x-\mathbf{j}_{p'}(s+t) & \text{if } x>\sign(s+t)|s+t|^{\frac{1}{p-1}}\\
    x-\mathbf{j}_{p'}(s-t) & \text{if } x< \sign(s-t)|s-t|^{\frac{1}{p-1}}\\    0 & \text{otherwise,}
    \end{cases}
\end{equation}
This new formulation is helpful to better show that as the forward gradient step is computed in the primal space: the pointwise map $\mathbf{j}_{p'}$ is required to associate the element $s+t$ of the dual space to an element in the primal space, thus making the computation well defined. In agreement with what has been discussed in the introduction, see \eqref{eq:primal_algo}, we can thus interpret this proximal-gradient iteration as a \emph{primal} algorithmic iteration.

As far as the thresholding function  $T^{\text{Alg.D}}(\cdot,s,t,p)$ in 
\eqref{eq:thresholding2} is concerned, we notice that we can proceed similarly and rewrite it as
{
\begin{align} 
\label{eq:thresholding2_new}
    T^{\text{Alg.D}}(x,s,t,p) & =
    \begin{cases}
    (\mathbf{j}_p(x)-s-t)^{\frac{1}{p-1}} & \text{if } \, |x|^{p-1}\sign(x)>s+t\\
    -(s-t-\mathbf{j}_p(x))^{\frac{1}{p-1}} & \text{if } \, |x|^{p-1}\sign(x)<s-t\\
    0 & \text{otherwise.}
    \end{cases}   \notag \\
    & =
    \begin{cases}
    \mathbf{j}_{p'}(\mathbf{j}_p(x)-s-t) & \text{if } \, |x|^{p-1}\sign(x)>s+t\\
    -\mathbf{j}_{p'}(s-t-\mathbf{j}_p(x)) & \text{if } \, |x|^{p-1}\sign(x)<s-t\\
    0 & \text{otherwise.}
    \end{cases}   
\end{align}
}
We can observe that both pointwise maps $\mathbf{j}_{p}$ and $\mathbf{j}_{p'}$ associated to the duality maps $\mathbf{J}^r_{p}$ and  $\mathbf{J}^{r'}_{p'}$, respectively,  play a role here as both a primal-to-dual map is required to define the forward step in the dual space and a dual-to-primal one is needed to compute the backward step, analogously to what we have seen before. In agreement with what has been observed in the introduction, see \eqref{eq:dual_algo}, we thus interpret such iteration as a \emph{dual} proximal-gradient step.

Such considerations can now be similarly applied to our modular-based variable exponent case, where the role of the primal-to-dual map $\mathbf{J}_{p}^r$ is replaced by the map $\jrhobar$,  whose expression in $\ell^{p(\cdot)}(\mathbb{R}^n)$ reads
\begin{equation*}
   \jrhobar(x) =
   \left(\sign(x_i)|x_i|^{p_i-1}\right)_{i=1\ldots,n} = \left( \mathbf{j}_{p_i}(x_i)\right)_{i=1\ldots,n},
\end{equation*}
and depends on the pointwise mappings $\mathbf{j}_{p_i}(\cdot)$ defined above. Algorithm \ref{alg_bredies} should thus be interpreted as a \emph{primal} proximal-gradient algorithm with pointwise thresholding function \eqref{eq:thresholding1_new}, while Algorithm \ref{alg_guansong} should be understood as a \emph{dual} proximal-gradient algorithm with pointwise thresholding function \eqref{eq:thresholding2_new}.
\section{Numerical tests}
\label{sec:tests} In this section, we provide some numerical tests showing how the proposed model adapts to deal with a variety of signal and image deconvolution and denoising problems. We include further tests providing a numerical verification of the computational convergence properties of Algorithms \ref{alg_bredies} and \ref{alg_guansong}.

\subsection{Spike reconstruction}
\label{sec:example1}

As a first example, we consider a 1D signal reconstruction problem where we seek for sparse spikes defined on $\Omega=[0,1]$ to be reconstructed from blurred measurements corrupted with Gaussian noise, see Figure \ref{fig:bredies(a)}. 
In order to favor sparse reconstructions and reduce possible oversmoothing, the formulation of a reconstruction model in a Banach space $\mathcal{X}$ is considered, see, e.g.,  \cite{Bredies2008,Schuster2012}. 
Denoting by $A:\mathcal{X}\longrightarrow L^{2}(\Omega)$ the blurring operator and by $y\in {L^2(\Omega)}$ the measured data, we thus aim to solve 
\begin{equation}
    \argmin_{x\in\mathcal{X}} \frac{1}{2}\|Ax-y\|_{2}^{2} +\lambda \|x\|_1,\qquad \lambda>0
    \label{var_model_2}
\end{equation}
where, in particular, we set $\mathcal{X}=\lpvar$ with $p_-=1.6$ and $p_+=2$, as shown in orange in Figure \ref{fig:bredies(a)}. The idea is to choose a higher value of $p(\cdot)$ where it is more likely to have a (spike) signal, while lower values are preferred elsewhere, so that, for these points, sparsity is enforced to a stronger extent. Note that the choice of the exponent map $p(\cdot)$ acts in fact as a prior model on the signal, together with the penalty term. To incorporate such prior knowledge, one can look directly at the shape of the data $y$, or, for instance, to the structure of a standard $\ell_2-\ell_1$ reconstruction computed after a small number of iterations, in order to have a variable exponent $p(\cdot)$ consistent with an approximated (possibly oversmoothed) solution of the problem, see \cite{Estatico2018} for more details on the choice of $p(\cdot)$.
Having chosen the exponent map $p(\cdot)$, we can thus solve \eqref{var_model_2} on $\mathcal{X}=\lpvar$ by means, e.g., of Algorithm \ref{alg_bredies}.

\begin{figure}[t!]
    \centering
    \begin{subfigure}[b]{0.325\textwidth}
    \includegraphics[width=\textwidth]{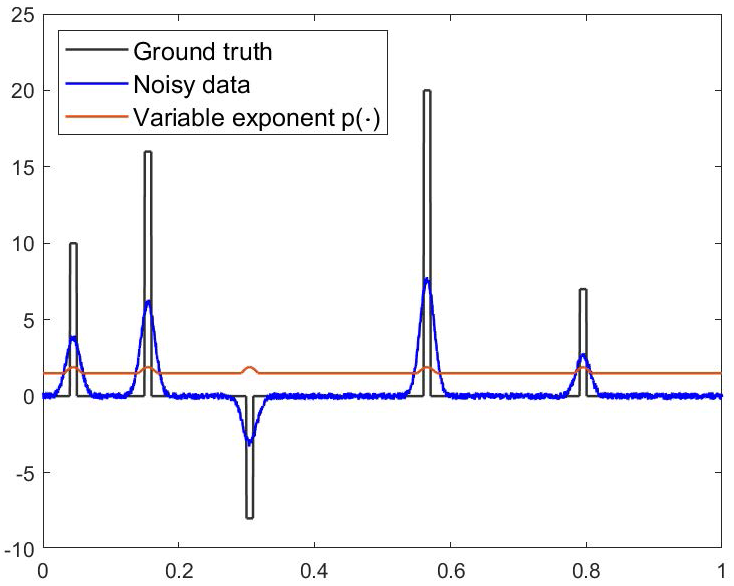}
    \caption{Data and $p(\cdot)$ map}
    \label{fig:bredies(a)}
    \end{subfigure}
    \begin{subfigure}[b]{0.325\textwidth}
    \includegraphics[width=\textwidth]{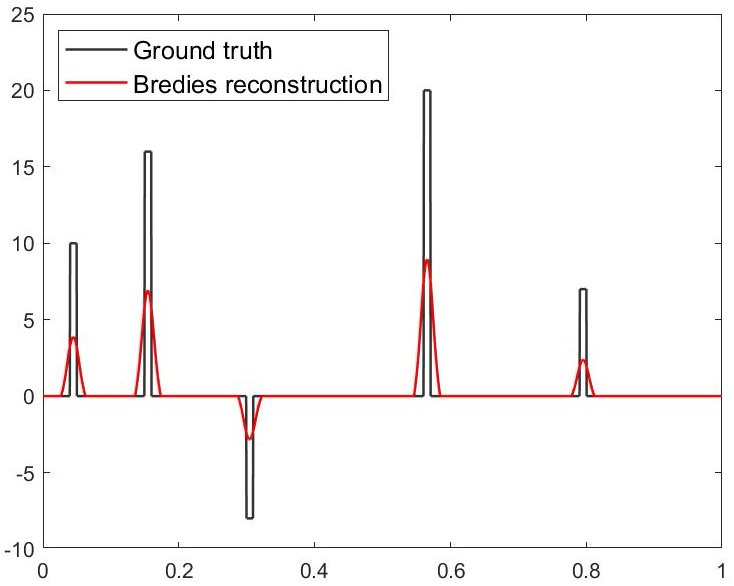}
    \caption{$L^{1.7}(\Omega)$ reconstruction}
    \label{fig:bredies(c)}
    \end{subfigure}
    \begin{subfigure}[b]{0.325\textwidth}
    \includegraphics[width=\textwidth]{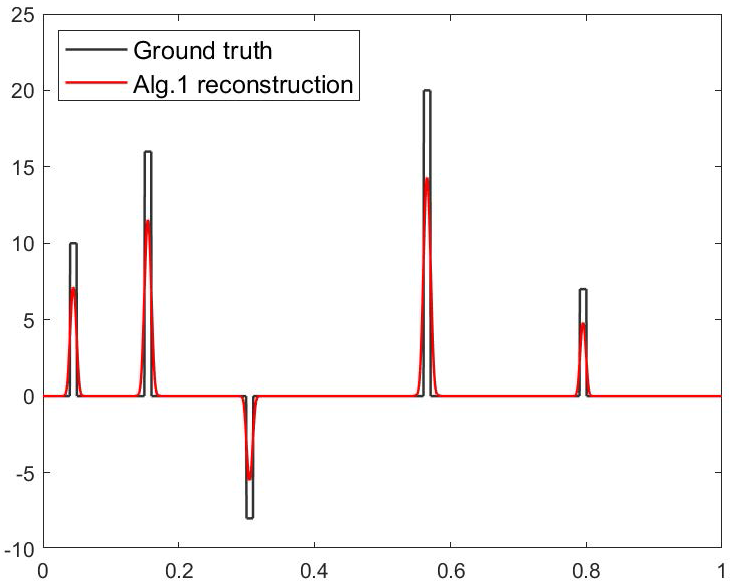}
    \caption{$\lpvar$ reconstruction}
    \label{fig:bredies(d)}
    \end{subfigure}
    \caption{Parameters: $\tau_k\equiv 0.5$; $\lambda=10^{-2}$. Stopping criterion based on the normalized relative change between $\xk$ and $\xkk$: $\|\xk-\xkk\|_2/\|\xk\|_2<10^{-4}$.}
\label{fig:bredies}
\end{figure}

In order to provide a comparison with existing models, we further consider problem $\eqref{var_model_2}$ on  $\mathcal{X}=L^{p}(\Omega)$ for $p=1.7$ and solve it by means of the algorithm in \cite{Bredies2008}. 
We observe that using  $\lpvar$ modeling improves the quality of the reconstruction with respect to a fixed $\lp$ modeling. 

\subsection{Deconvolution of heterogeneous signals}
\label{sec:segnalemisto}

We now consider the signal deconvolution problem of a blurred and noisy data $y\in{L^2(\Omega)}$ corrupted by Gaussian noise of a 1D heterogeneous signal $x\in\mathcal{X}$ with $\Omega=[0,1]$ composed of spikes and a smooth part, as shown in Figure \ref{fig:sec6_2(f)}. We consider the $\ell_2$ norm of the residual as data term, so that  model \eqref{var_model_2} is used as reconstruction criterion
with blurring operator $A:\mathcal{X}\longrightarrow L^2(\Omega)$.

\begin{figure}[t]
\centering
\begin{subfigure}[b]{0.325\textwidth}
\includegraphics[width=\textwidth]{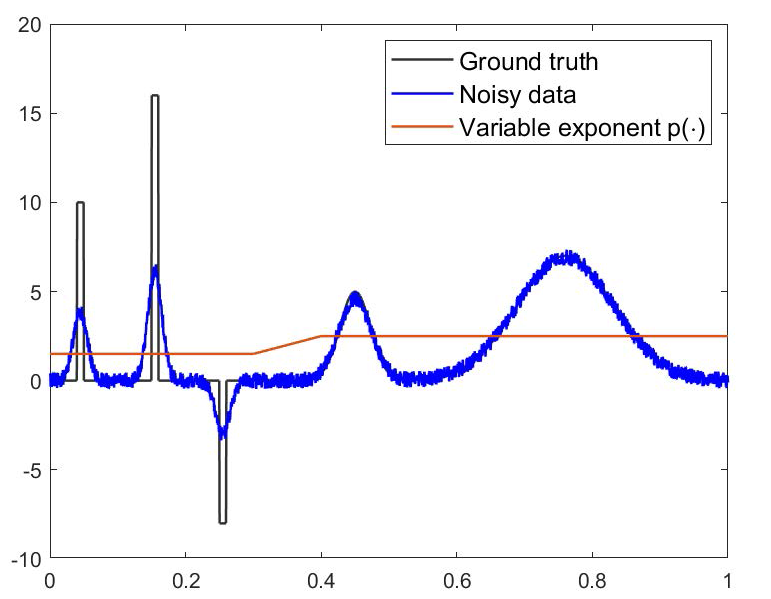}
\caption{Data and $p(\cdot)$ map}
\label{fig:sec6_2(f)}
\end{subfigure}
\begin{subfigure}[b]{0.325\textwidth}
\includegraphics[width=\textwidth]{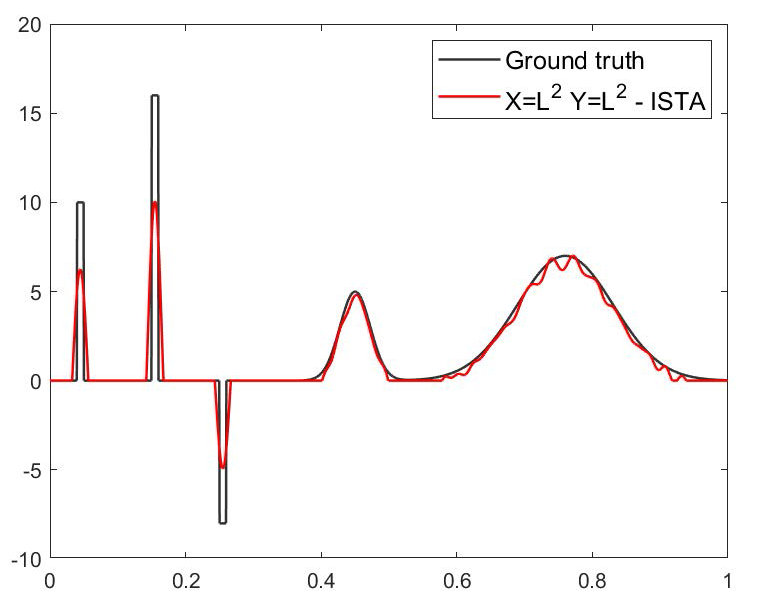}
\caption{$A: L^2\rightarrow L^2$, ISTA}
\label{fig:sec6_2(a)}
\end{subfigure}\\
\begin{subfigure}[b]{0.325\textwidth}
\includegraphics[width=\textwidth]{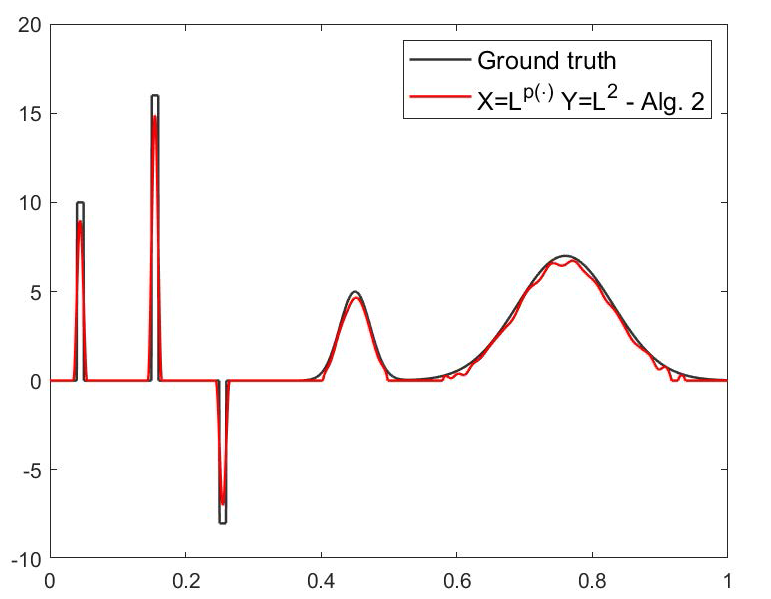}
\caption{$A: L^{p(\cdot)}\rightarrow L^2$, Alg. \ref{alg_guansong}}
\label{fig:sec6_2(b)}
\end{subfigure}
\begin{subfigure}[b]{0.325\textwidth}
\includegraphics[width=\textwidth]{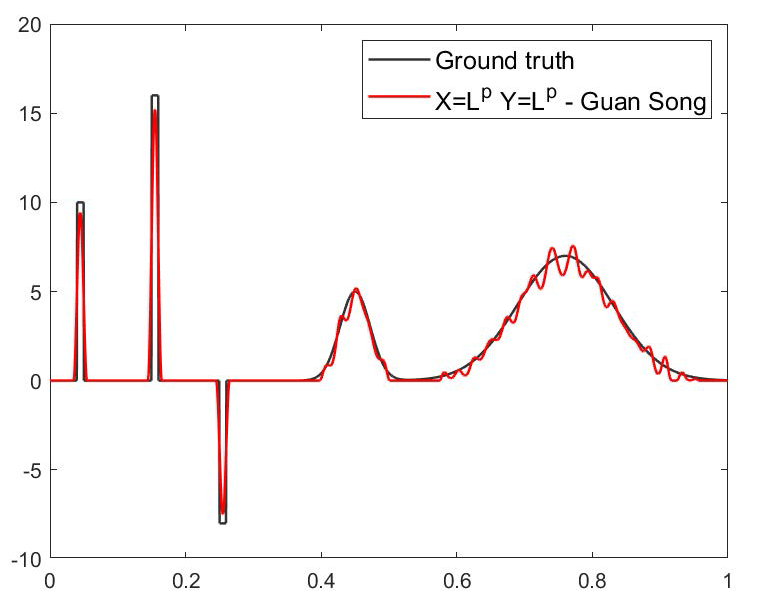}
\caption{$A: L^{1.5}\rightarrow L^{1.5}$, Alg. \cite{GuanSong2015}}
\label{fig:sec6_2(c)}
\end{subfigure}
\begin{subfigure}[b]{0.325\textwidth}
\includegraphics[width=\textwidth]{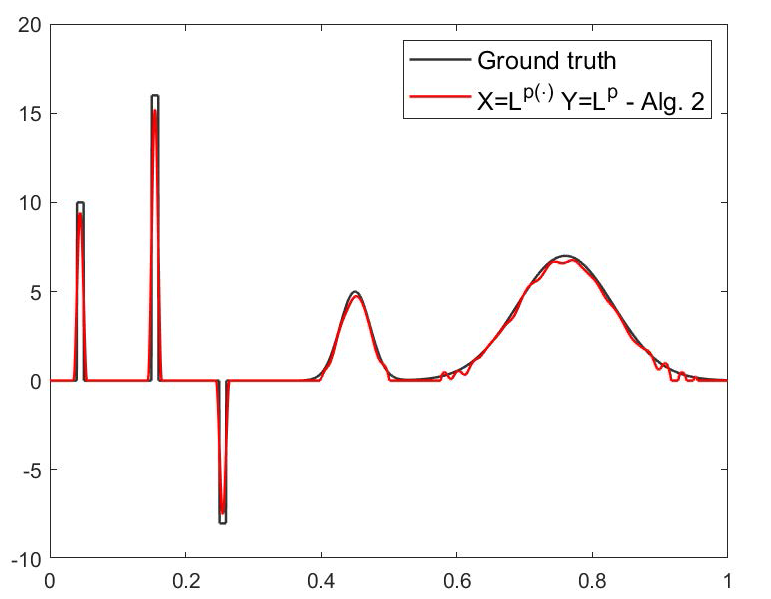}
\caption{$A: L^{p(\cdot)}\rightarrow L^{1.5}$, Alg. \ref{alg_guansong}}
\label{fig:sec6_2(d)}
\end{subfigure}
\caption{The variable exponent $1.5=p_-\le p(\cdot)\le 2=p_+$ considered in Figures \ref{fig:sec6_2(b)}, \ref{fig:sec6_2(c)} and \ref{fig:sec6_2(d)} is shown in orange in Figure \ref{fig:sec6_2(f)}. Parameters: $\tau_k\equiv 0.5$; $\lambda=5*10^{-3}$. Stopping criterion based on the relative distance between $\xk$ and $\xkk$: $\|\xk-\xkk\|_2/\|\xk\|_2<4 * 10^{-6}$.}
\label{fig:sec6_2}
\end{figure}

In Figures \ref{fig:sec6_2(a)} and \ref{fig:sec6_2(b)} we report the reconstructions obtained by solving \eqref{var_model_2} with solution spaces $\mathcal{X}=L^2(\Omega)$ using ISTA algorithm \cite{Daubechies2003}, and $\mathcal{X}=L^{p(\cdot)}(\Omega)$ using Algorithm \ref{alg_guansong}, respectively, for the particular choice of exponent map $p(\cdot)$ having $p_-=1.5$ and $p_+=2$ shown in orange in Figure \ref{fig:sec6_2(f)}. We observe that the spikes in the left-hand side are better reconstructed when considering $\lpvar$ as solution space, due to its locally enhanced sparsifying property.

Similarly as in \cite{Bredies2008}, we could also consider an $\ell_p$, $1<p<2$, fidelity to better restore spikes. Consistently, we can consider $A:\mathcal{X}\longrightarrow L^p(\Omega)$ with $p=1.5$ and the following associated variational model
\begin{equation}
    \argmin_{x\in\mathcal{X}} \frac{1}{p}\|Ax-y\|_{p}^{p} +\lambda \|x\|_1,\qquad \lambda>0.
    \label{var_model_2_bis}
\end{equation}
Figures \ref{fig:sec6_2(c)} and \ref{fig:sec6_2(d)} show the reconstructions obtained with $\mathcal{X}=L^{1.5}(\Omega)$ using the algorithm proposed in \cite{GuanSong2015} and $\mathcal{X}=L^{p(\cdot)}(\Omega)$ using Algorithm \ref{alg_guansong} respectively. We observe that smooth regions are better restored with reduced ringing effect artifacts in Figure \ref{fig:sec6_2(d)} thanks to the flexible choice of the solution space.

Working with a variable exponent allows us to deal with the different nature of the signal in a more flexible way.

\subsection{1D and 2D mixed noise removal}
\label{sec:rumoremisto}

We now focus on a mixed noise removal problem for blurred signals and images affected by Gaussian and impulsive (salt and pepper) noise, in different and disjoint parts of their spatial domain $\Omega$, with $\Omega=[0,1]$ and $\Omega$ being a compact of $\R^2$, respectively. We exploit here the flexibility of Lebesgue spaces with variable exponent by treating effectively the different noise nature at the same time. 

For the 1D example, the measured blurred and noisy signal is shown in Figure \ref{fig:sec6_3_1d(a)}:  Gaussian noise is artificially added on the left part of the domain while impulsive noise is added on the right part.
Let $\mathcal{X},\mathcal{Y}$ be two Lebesgue spaces with variable exponent maps $p^{\mathcal{X}}(\cdot), p^{\mathcal{Y}}(\cdot)$, respectively, and let $A\in\mathcal{L}(\mathcal{X},\mathcal{Y})$. To retrieve the sparse underlying signal, let us consider the following problem:
\begin{equation}
    \argmin_{x\in\mathcal{X}} \bar{\rho}_{p^\mathcal{Y}(\cdot)}(Ax-y)+\lambda \|x\|_1, \qquad \lambda>0.
    \label{eq1}
\end{equation}

The blurred signal $y\in\mathcal{Y}$ is corrupted by Gaussian noise on the left-hand side and by impulsive noise on the right-hand side (in blue in Figure \ref{fig:sec6_3_1d(a)}). Choosing $p^\mathcal{X}(\cdot)=p^\mathcal{Y}(\cdot)\equiv 2$ in \eqref{eq1} (and thus naturally setting $\mathcal{X}=\mathcal{Y}=L^2(\Omega)$) forces the fidelity term to reduce to $\|Ax-y\|^2_2$ which is well-known to be the most appropriate term to describe Gaussian noise degradation. Hence, in this case, an $\ell_2-\ell_1$ variational model formulated in $L^2(\Omega)$ is obtained. For solving it numerically, the standard ISTA algorithm \cite{Daubechies2003} can be used. The corresponding reconstruction is shown in Figure \ref{fig:sec6_3_1d(b)}: we note that the reconstruction on the left-hand side is good, while several artifacts can be observed on the right-hand side, due to the poor adaptivity of the model to the different noise nature there. A data term more adapted to describing the sparse nature of impulsive noise should be considered, such as, ideally, an $\ell_1$ fidelity; see \cite{Nikolova2004}.  In our modeling, however, exponents $p=1$ cannot be chosen as they would correspond to a nonsmooth data term defined in a nonreflexive Banach space. However, exponents $p\approx 1$ can still be chosen. We thus consider as output space a Lebesgue space $\mathcal{Y}$ with variable exponent map $p^\mathcal{Y}(\cdot)=p(\cdot)$ shown in orange in Figure \ref{fig:sec6_3_1d(a)}: it is equal to 2 in the part of the domain where the signal is corrupted by Gaussian noise and equal to 1.4 elsewhere. The space-variant fidelity defined in this way is, therefore, differentiable and locally adapted to both natures of the noise. In Figure \ref{fig:sec6_3_1d(c)} we show the ISTA reconstruction obtained by choosing the solution space as $\mathcal{X}=L^2(\Omega)$: our choice of fidelity and measurement space $\mathcal{Y}$ is better adapted to both noise distributions and favors the removal of impulsive noise from data, but the spikes on the right-hand side suffer from some intensity loss. To improve upon this drawback, the solution space $\mathcal{X}=\lpvar$ with the same choice of variable exponent map considered for the output space can be considered. In Figure \ref{fig:sec6_3_1d(d)}, we show the reconstruction obtained by solving \eqref{eq1} with $\mathcal{X}=\mathcal{Y}=\lpvar$ via Algorithm \ref{alg_guansong}. The choice of an adaptive solution space improves the quality of the reconstruction and endows the model with enough flexibility to provide a good reconstruction of the signal over the whole domain.

A final comparison with  $\mathcal{Y}=\lp$ with $p=1.4$ in \eqref{eq1} is reported in Figures \ref{fig:sec6_3_1d(g)} and \ref{fig:sec6_3_1d(f)}. They show the reconstructions obtained by solving \eqref{eq1} with $\mathcal{Y}=L^{1.4}(\Omega)$ and solution spaces $\mathcal{X}=\lpvar$ and $\mathcal{X}=L^{1.4}(\Omega)$, respectively. 
Note that the reconstruction in Figure \ref{fig:sec6_3_1d(d)} is more accurate than the one in Figure \ref{fig:sec6_3_1d(f)}, in particular  on the left-hand side since the  spikes there are better reconstructed. 

\begin{figure}[t!]
\centering
\begin{subfigure}[b]{0.325\textwidth}
\includegraphics[width=\textwidth]{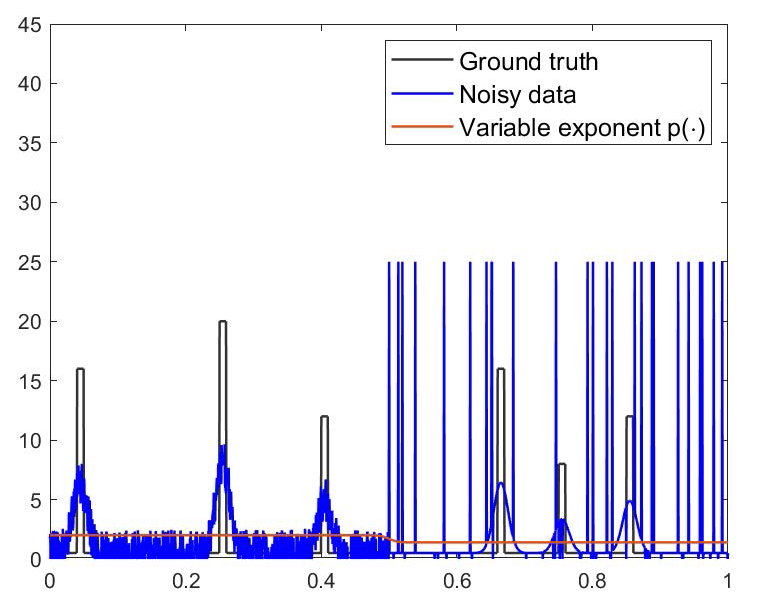}
\caption{Data and $p(\cdot)$ map}
\label{fig:sec6_3_1d(a)}
\end{subfigure}
\begin{subfigure}[b]{0.325\textwidth}
\includegraphics[width=\textwidth]{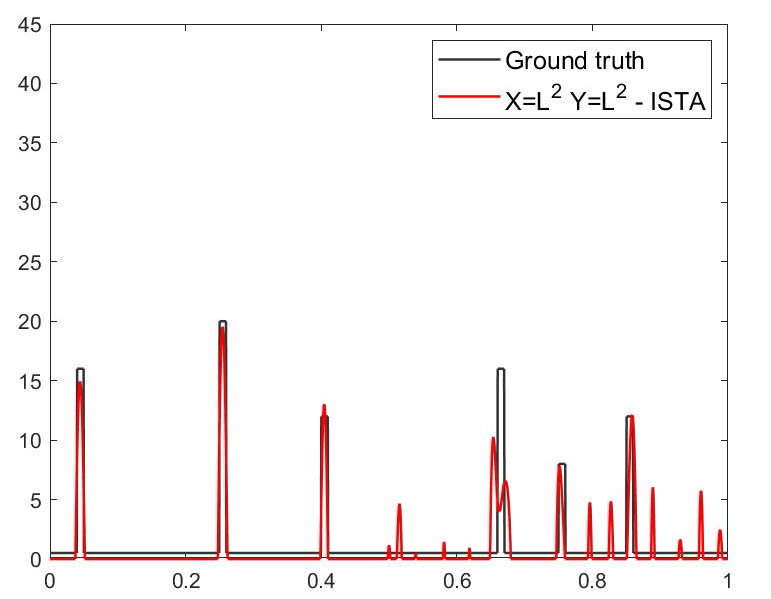}
\caption{$A: L^2\rightarrow L^2$, ISTA}
\label{fig:sec6_3_1d(b)}
\end{subfigure}
\begin{subfigure}[b]{0.325\textwidth}
\includegraphics[width=\textwidth]{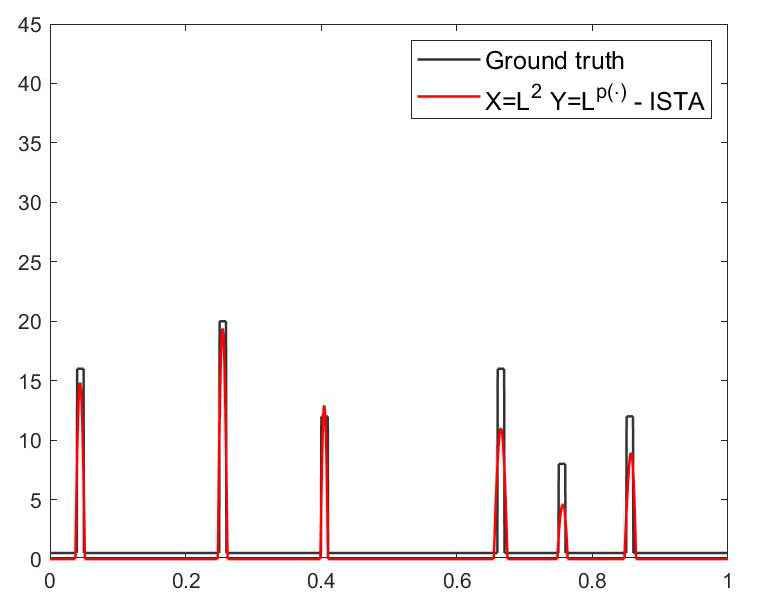}
\caption{$A: L^2\rightarrow L^{p(\cdot)}$, ISTA}
\label{fig:sec6_3_1d(c)}
\end{subfigure}\\
\begin{subfigure}[b]{0.325\textwidth}
\includegraphics[width=\textwidth]{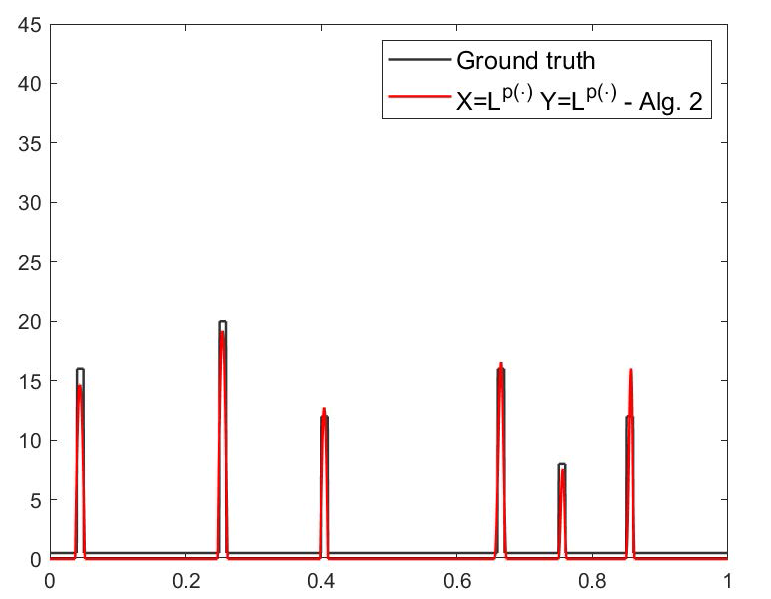}
\caption{$A: L^{p(\cdot)}\rightarrow L^{p(\cdot)}$, Alg. \ref{alg_guansong}}
\label{fig:sec6_3_1d(d)}
\end{subfigure}
\begin{subfigure}[b]{0.325\textwidth}
\includegraphics[width=\textwidth]{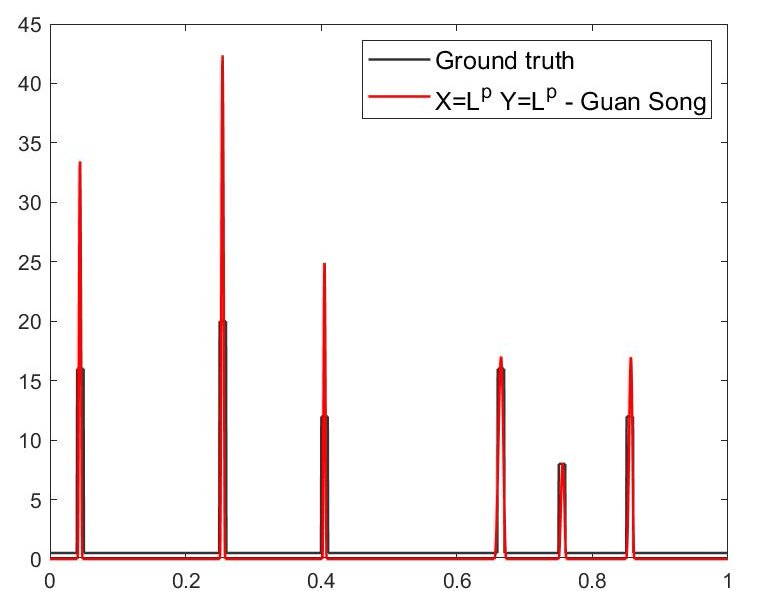}
\caption{$A: L^{1.4}\rightarrow L^{1.4}$, Alg. \cite{GuanSong2015}}
\label{fig:sec6_3_1d(g)}
\end{subfigure}
\begin{subfigure}[b]{0.325\textwidth}
\includegraphics[width=\textwidth]{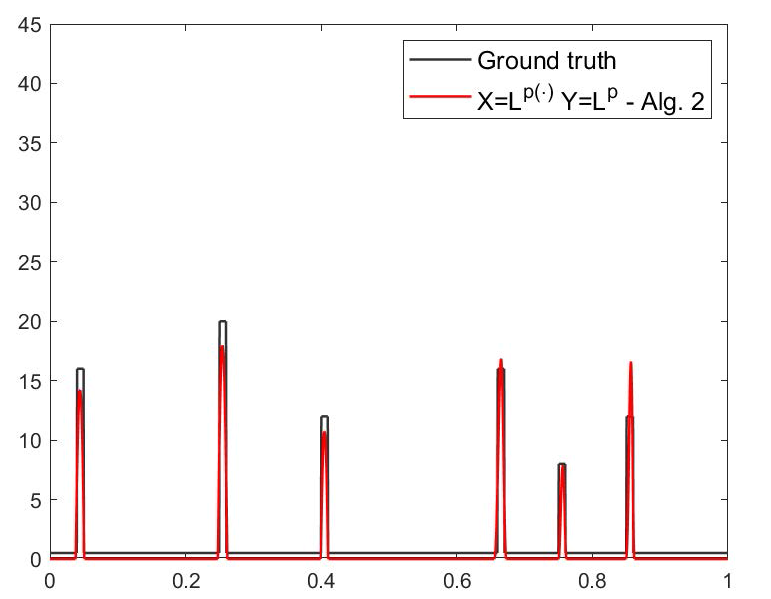}
\caption{$A: L^{p(\cdot)}\rightarrow L^{1.4}$, Alg. \ref{alg_guansong}}
\label{fig:sec6_3_1d(f)}
\end{subfigure}
\caption{The variable exponent $1.4=p_-\le p(\cdot)\le 2=p_+$ considered in Figures \ref{fig:sec6_3_1d(c)}, \ref{fig:sec6_3_1d(d)}  and \ref{fig:sec6_3_1d(f)} is shown in orange in Figure \ref{fig:sec6_3_1d(a)}. Parameters:  $\tau_k\equiv 0.1$, $\lambda=2*10^{-2}$.  Stopping criterion based on the relative distance between $\xk$ and $\xkk$: $\|\xk-\xkk\|_2/\|\xk\|_2<4*10^{-6}$.}
\label{fig:sec6_3_1d}
\end{figure}
\begin{figure}[h!]
    \centering
    \begin{subfigure}[b]{0.24\textwidth}
    \includegraphics[width=\textwidth]{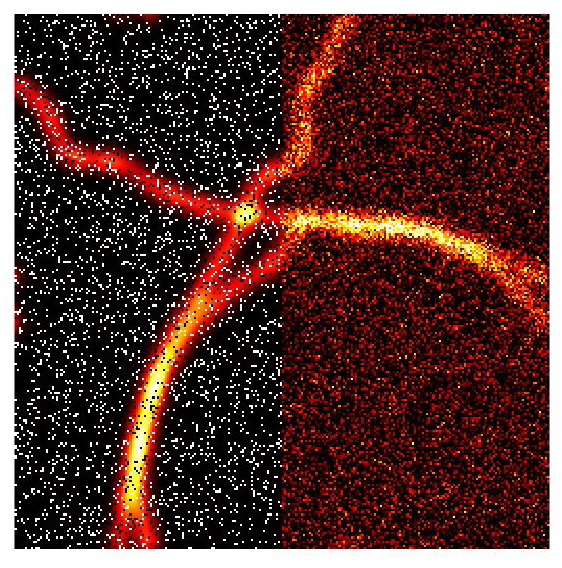}
    \caption{Noisy data}
    \label{fig:filamento(a)}
    \end{subfigure}
    \begin{subfigure}[b]{0.24\textwidth}
    \includegraphics[width=\textwidth]{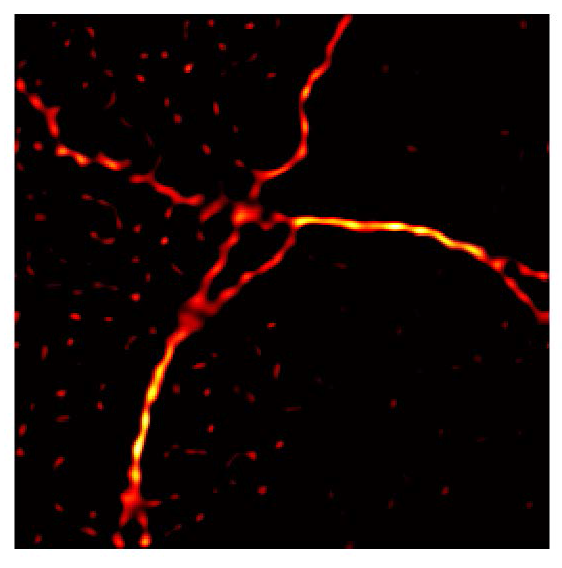}
    \caption{$\mathcal{X}=\mathcal{Y}=L^2(\Omega)$}
    \label{fig:filamento(b)}
    \end{subfigure}
    \begin{subfigure}[b]{0.24\textwidth}
    \includegraphics[width=\textwidth]{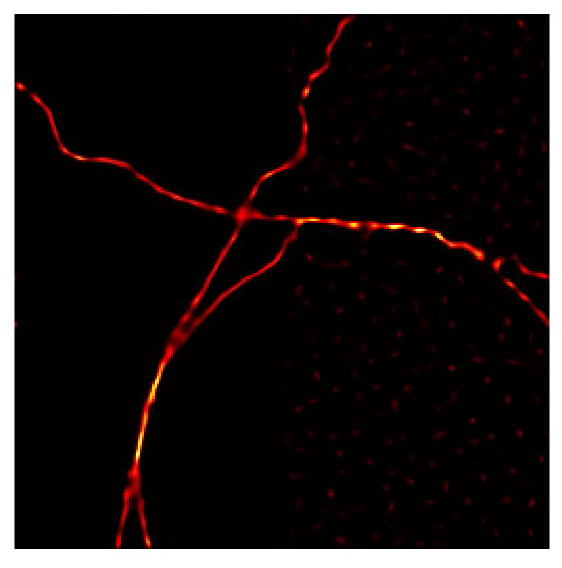}
    \caption{$\mathcal{X}=\mathcal{Y}=L^{1.4}(\Omega)$}
    \label{fig:filamento(c)}
    \end{subfigure}
    \begin{subfigure}[b]{0.24\textwidth}
    \includegraphics[width=\textwidth]{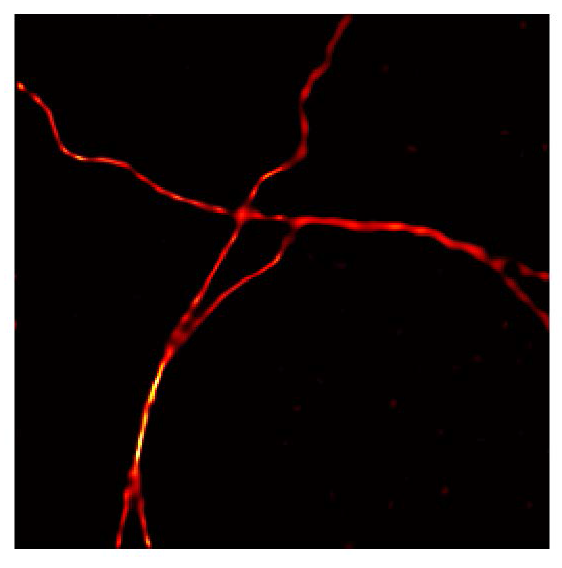}
    \caption{$\mathcal{X}=\mathcal{Y}=L^{p(\cdot)}(\Omega)$}
    \label{fig:filamento(d)}
    \end{subfigure}
    \caption{Parameters: $\tau_k\equiv 0.1$, $\lambda=0.1$. Stopping criterion based on the normalized relative change between $\xk$ and $\xkk$: $\|\xk-\xkk\|_2/\|\xk\|_2<10^{-4}$. }
\label{fig:filamento}
\end{figure}
As a similar test, we considered a mixed denoising problem for the blurred and noisy image in Figure \ref{fig:filamento(a)}, again with Gaussian noise on the left half and impulsive noise on the right half of the image domain. 
We solved again \eqref{eq1} for $p^\mathcal{X}(\cdot)=p^\mathcal{Y}(\cdot)\equiv 2$ by the ISTA algorithm, for $ p^\mathcal{X}(\cdot)=p^\mathcal{Y}(\cdot)\equiv p=1.4$ by the algorithm in \cite{GuanSong2015}, and for $ p^\mathcal{X}(\cdot)=p^\mathcal{Y}(\cdot)=p(\cdot)$, where $p(\cdot)$ is a variable exponent $p_-=1.4\leq p(\cdot) \leq p_+=2$  such that $p_i=2$ on the Gaussian noisy half and $p_i=1.4$ on the impulsive noisy half, by Algorithm \ref{alg_guansong}.
Observations analogous to those discussed for the 1D case can still be made and as one can clearly see, reconstruction artifacts are significantly reduced in the case of a variable exponent modeling. 

The flexibility of the model given by the choice of the map $p(\cdot)$ allows one to reconstruct signals with different spatial properties on the whole domain.

\subsection{A numerical study on convergence rates}
\label{sec:conv_rates}


The previous examples show that the use of a variable exponent can help in improving reconstruction quality. It is thus natural to ask which algorithm -- Algorithm \ref{alg_bredies} and Algorithm \ref{alg_guansong} -- should be used in practice. As remarked already in  \cite{Bredies2008} and as it can be observed from the convergence rate \eqref{eq:conv_Bredies},  Algorithm \ref{alg_bredies} is expected to be very slow in practice, in particular, slower than a gradient-type algorithm whose well-known convergence speed is of the order $O(1/k)$. 

In this section, we compare the speed of convergence of different algorithms when used as numerical solvers for the deblurring problem \eqref{var_model_2}  in the different Hilbert and Banach scenarios discussed in Section \ref{sec:example1}. In particular, we compare the speed of convergence for the ISTA algorithm used to solve \eqref{var_model_2}  in $\mathcal{X}=L^2(\Omega)$, with the algorithms proposed by Bredies in \cite{Bredies2008} and Guan and Song in \cite{GuanSong2015} for solving the same problem on $\mathcal{X}=\lp$ with $p=1.7$ and with Algorithms \ref{alg_bredies} and \ref{alg_guansong} proposed in this work for $\mathcal{X}=\lpvar$. As exponent map $p(\cdot)$, we stick with the choice shown in orange Figure \ref{fig:bredies(a)}. 

Given $x^*\in\mathcal{X}$, solution of \eqref{var_model_2} in the different spaces $\mathcal{X}$, we recall the convergence rates \eqref{eq:conv_Bredies} and \eqref{eq:conv_GS} for Algorithms \ref{alg_bredies} and \ref{alg_guansong}, respectively. In principle, to compare the speed of convergence in a precise way, a precomputation of $x^*$ by means of benchmark algorithms should be done for all the different scenarios discussed above. However, since to our knowledge there is no existing algorithm for solving \eqref{var_model_2} in $\mathcal{X}=\lpvar$, instead of computing $x^*$ we computed as a reference the value of $\tilde{x}$, solution of \eqref{var_model_2} with $\mathcal{X}=L^2(\Omega)$, by running ISTA for $2*10^4$ iterations. Note that by simple manipulations we have for Algorithm \ref{alg_bredies} 
\begin{equation*}
    \phi(\xk)-\phi(\Tilde{x})=\phi(\xk)\pm\phi(x^*)-\phi(\Tilde{x})\le \eta_1 \Big(\frac{1}{k}\Big)^{p_--1} + c 
\end{equation*}
and, similarly, for Algorithm \ref{alg_guansong} $\phi(\xk)-\phi(\Tilde{x})\le \eta_2 \Big(\frac{1}{k}\Big)^{\texttt{p}-1} +c$ with $\texttt{p}=2$ in \eqref{var_model_2}, so that rates can still be compared up to an additive constant. We use $\Tilde{x}$ also for the computation of the convergence rates for the Bredies \cite{Bredies2008} and Guan-Song \cite{GuanSong2015} algorithms having  $\phi(\xk)-\phi(\Tilde{x})\le \eta_{3,4} \Big(\frac{1}{k}\Big)^{\texttt{p}-1}+c'$. Finally, recall that for ISTA the convergence rate in function values is $\phi(\xk)-\phi(\Tilde{x})\le \eta_5 \frac{1}{k}$.
As stopping criterion, for all the tested algorithms we use the normalized relative rates with respect to $\Tilde{x}$ up to a tolerance parameter $\epsilon=10^{-4}$: $|\phi(\xk)-\phi(\Tilde{x})|/\phi(\Tilde{x})<\epsilon$. 

\begin{figure}[h!]    
    \centering
    \begin{subfigure}[b]{0.45\textwidth}
    \includegraphics[width=\textwidth]{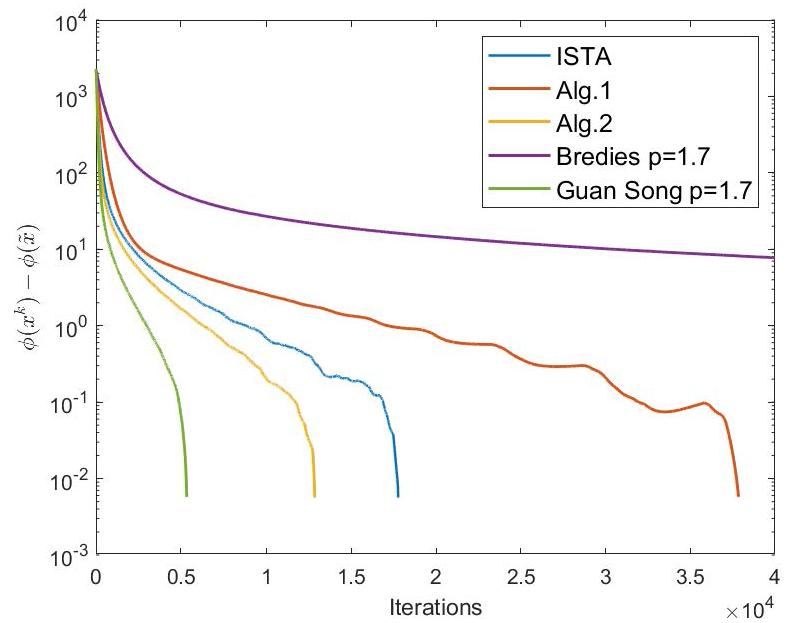}
    \caption{Iterations}
    \label{fig:opt(a)}
    \end{subfigure}
    \begin{subfigure}[b]{0.45\textwidth}
    \includegraphics[width=\textwidth]{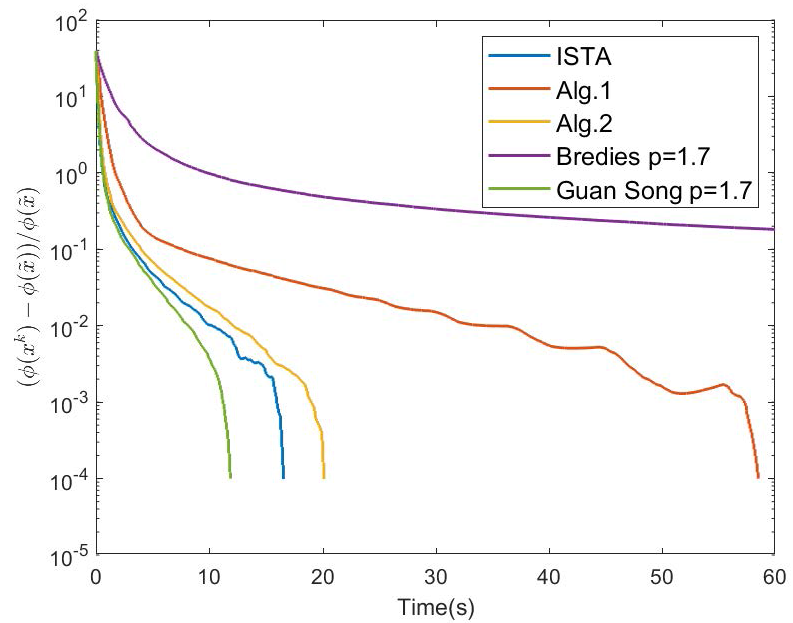}
    \caption{CPU times}
    \label{fig:opt(b)}
    \end{subfigure}
    \caption{(a) Relative rates along the first $4*10^4$ iterations. (b) Normalized relative rates along the first 60 seconds of CPU time.}
    \label{fig:opt}
\end{figure}
\begin{table}[h!]
\begin{tabular}{l|l|l|l|l|l|}
\cline{2-6}
                       & ISTA & Bredies \cite{Bredies2008} & Guan-Song \cite{GuanSong2015} & Alg. \ref{alg_bredies} & Alg. \ref{alg_guansong} \\ \hline
 \multicolumn{1}{|l|}{\# iterations} & 17768 & $5*10^5$ & 5352 & 37850 & 12849 \\ \hline
\multicolumn{1}{|l|}{CPU time} & 16.5s & 1025.7s & 11.9s & 58.5s & 20.0s \\ \hline
\end{tabular}
\caption{Algorithmic comparison: iterations required and CPU time till convergence.}
\label{tab:opt}
\end{table}

The results reported in Figure \ref{fig:opt}  and Table \ref{tab:opt} show several interesting numerical convergence properties. First, we note that although the Bredies (violet line) and Guan-Song (green line) algorithms are supposed to have the same convergence rate $O(1/k^{0.7})$ in theory for the specific problem at hand, they clearly have a very different  behavior. While the first needs more than $5*10^5$ iterations and more than $1000$ seconds of CPU time to reach convergence, the second converges with a much faster speed. The same behavior is observed also for the modular Algorithms \ref{alg_bredies} and \ref{alg_guansong} too, with the first (red line) being very slow and the second (yellow line) much faster. With respect to standard ISTA, we observe that the proposed Bregmanized algorithm and the Guan-Song one are faster in terms of number of iterations and comparable in terms of computational time. From these numerical verifications, we thus considered Algorithm \ref{alg_guansong} instead of Algorithm \ref{alg_bredies} in sections \ref{sec:segnalemisto} and \ref{sec:rumoremisto}. The acceleration of these algorithms is indeed an interesting matter for future research.

\section{Conclusions}
In this paper, we present two possible ways of generalizing proximal-gradient algorithms to solve structured minimization problems in $\lpvar$. Due to the difficulties introduced by working in a Banach space setting (i.e., the lack of Riesz isometric isomorphism and the need for duality mappings) and the further lack of separability of the norm in $\lpvar$, standard strategies cannot be applied here. Thus the use of alternative iterative schemes based on the separable modular function is required. Two algorithms are presented, a primal one (Algorithm \ref{alg_bredies}) inspired by the one studied in \cite{Bredies2008} for problems in $\lp$, and a dual one (Algorithm \ref{alg_guansong}), analogous to the one studied in \cite{GuanSong2015}.
We provide a detailed convergence analysis for Algorithm \ref{alg_bredies} showing the descent of the functional and convergence in function values with rate $\mathcal{O}\Big(\frac{1}{k}\Big)^{p_--1}$, with $p_-$ being the essential infimum of the exponent map $p(\cdot)$. For Algorithm \ref{alg_guansong} a modular Bregman-like distance is used in a dual update fashion and analogous convergence rates are shown in function values. Several numerical results are reported to compare the use of the adaptive $\lpvar$ space with  Hilbert and $L^p(\Omega)$ spaces on some exemplar signal deconvolution and signal/image mixed noise removal problems. Finally, a numerical verification on the speed of convergence and on the computational efficiency of the proposed algorithms is given, showing that \textit{primal} algorithms have slower convergence in comparison with \textit{dual} ones. 

Future work shall address the crucial question of how the map  $p(\cdot)$ shall be chosen in order to adapt to local signal/image prior contents. Moreover, it would be interesting to find a strategy to incorporate extrapolation/acceleration in order to improve convergence speed which, in practice, is particularly slow for primal algorithms.

\bibliographystyle{plain}
\bibliography{ref.bib}
\end{document}